\documentclass[11pt]{amsart} \textwidth=14.5cm \oddsidemargin=1cm
\usepackage{amsmath, amstext, amsbsy, amssymb, amscd, mathrsfs}
\usepackage{amsmath}
\usepackage{amsxtra}
\usepackage{amscd}
\usepackage{amsthm}
\usepackage{amsfonts}
\usepackage{amssymb}
\usepackage{eucal}
\usepackage{color}
\usepackage[all]{xy}
\usepackage{bbding}
\usepackage{pmgraph}
\usepackage{stmaryrd}
\usepackage{graphicx}
\usepackage{bm}
\usepackage[enableskew,vcentermath]{youngtab}

\newtheorem{theorem}{Theorem}[section]

\newtheorem{lemma}[theorem]{Lemma}
\newtheorem{proposition}[theorem]{Proposition}
\newtheorem{corollary}[theorem]{Corollary}

\theoremstyle{definition}
\newtheorem{definition}[theorem]{Definition}

\theoremstyle{remark}
\newtheorem{remark}[theorem]{Remark}

\definecolor{A}{rgb}{.75,1,.75}

\numberwithin{equation}{section}

\newcommand{\C}{\mathbb C}

\newcommand{\HCa}{{\mathcal{H}_{n,\bf A}^c}}

\newcommand{\ga}{\gamma}

\newcommand{\la}{\lambda}

\newcommand{\mc}{\mathcal}

\newcommand{\End}{{\rm End}}
\newcommand{\Hom}{{\rm Hom}}

\newcommand{\mHR}{\mathscr{H}_{n,R}(q)}
\newcommand{\mHA}{\mathscr{H}_{n,{\bf A}}(q)}
\newcommand{\mH}{\mathscr{H}_n(q)}

{\vskip-\lastskip\medskip
  \noindent
  {\em #1.}\enspace
  }%
{\qed\par\medskip
  }

\begin{document}

\title
[Mirabolic Hecke algebras]
{Mirabolic Hecke algebras, Schur-Weyl duality and Frobenius character formulas}
\author[Jinkui Wan]{Jinkui Wan}

\address{
School of Mathematical Sciences,
Shenzhen University,
Shenzhen, 518060, P.R. China. }
\email{wjk302@hotmail.com}


\maketitle

\begin{abstract}
We first introduce a new presentation for the mirabolic Hecke algebra $\mHR$ over an arbitrary commutative ring $R$ and derive a new basis. Based this presentation, specializing to the case of $\mH$ over the field $\mathbb{C}(q)$, we construct a basis for the cocenter of $\mH$, which facilitates the definition of its character table. We further establish a Schur--Weyl duality between $\mH$ and the quantum group $U_q(\mathfrak{gl}_r)$. As an application, we obtain Frobenius character formulas for the irreducible characters of $\mH$ within the ring of symmetric functions. Finally, we derive a recursive Murnaghan--Nakayama rule for the computation of the character table.
\end{abstract}

\setcounter{tocdepth}{1}
 \tableofcontents

\section{Introduction}

 \subsection{}

The systematic computation of symmetric group characters dates back to the remarkable work of Frobenius \cite{Fr} in 1900. His methodology was subsequently refined into the Murnaghan--Nakayama rule, providing a combinatorial framework for character values (cf. \cite[Chapter I, Example 9]{Mac}). In his study of the general linear groups, Schur \cite{Sch1, Sch2} established that the Frobenius construction can be recovered by a way of the fundamental reciprocity between the general linear group $GL_r(\mathbb{C})$ and the symmetric group $\mathfrak{S}_n$, now universally termed Schur--Weyl duality.

Decades later, Jimbo \cite{Ji} extended this duality to the quantum setting, relating $U_q(\mathfrak{gl}_r)$ with the Iwahori--Hecke algebra $\mc{H}_n(q)$. Motivated by Schur's work, Ram \cite{Ra} (see also King--Wybourne \cite{KW}) developed a ``Frobenius-type'' formula and a combinatorial rule for the characters of these Hecke algebras. Furthermore, the theory of character tables for Hecke algebras was systematically formulated by Geck and Pfeiffer \cite{GP1}. These investigations uncovered profound links between the representation theory of Hecke algebras, Hall--Littlewood symmetric functions as well as Kronecker products of symmetric group representations.  Ultimately, these results yield a $q$-analogue of the Murnaghan--Nakayama rule that smoothly recovers the classical character theory of $\mathfrak{S}_n$ in the limit as $q \to 1$.
 
 When $q$ is a power of a prime, the Iwahori--Hecke algebra $\mathcal{H}_n(q)$ associated with the symmetric group is realized as the centralizer algebra $e\mathbb{C}GL_n(q)e$, where $e=|\mathsf{B}|^{-1}\sum_{b\in \mathsf{B}}b$ is the idempotent corresponding to the Borel subgroup $\mathsf{B}$ of upper triangular matrices in $GL_n(q)$. 
The study of the mirabolic Hecke algebra $\mH$ was initiated by Solomon in \cite{So1, So2}. Conceptually, one considers the group $GA(V)$ of affine transformations on $V=\mathbb{F}_q^n$, which is isomorphic to the semidirect product $\mathbb{F}_q^n\rtimes GL_n(\mathbb{F}_q)$ (cf. \cite{WZ}). Equivalently, $GA_n(q)$ can be identified with the mirabolic subgroup of $GL_{n+1}(q)$ consisting of matrices of the form:
$$
GA_n(q)=\Bigg\{\begin{bmatrix}1&0\\ \alpha& g \end{bmatrix}\mid \alpha\in\mathbb{F}_q^n, g\in GL_n(q)\Bigg\}.
$$
By regarding the Borel subgroup $\mathsf{B} \subset GL_n(q)$ as a subgroup of $GA_n(q)$, the resulting algebra $e\mathbb{C}GA_n(q)e$ coincides with $\mH$ in the prime power case. Furthermore, Rosso \cite{Ro1} identified $\mH$ with the convolution algebra of $GL_n(q)$-equivariant complex-valued functions on the triple variety $GL_n(q)/\mathsf{B} \times GL_n(q)/\mathsf{B} \times \mathbb{F}_q^n$. This generalizes the classical realization of $\mathcal{H}_n(q)$ as the convolution algebra of $GL_n(q)$-equivariant functions on the double flag variety $GL_n(q)/\mathsf{B} \times GL_n(q)/\mathsf{B}$. Recently, the study of mirabolic structures has seen substantial progress; see, for instance, \cite{BFGT, FZM, FGT1, FGT2, FU, GR, Ro1, Ro2, Sh, SS, Tr}.

 \subsection{}
The primary objective of this article is to derive a ``Frobenius-type'' formula and an explicit combinatorial rule for the characters of $\mH$. To this end, we first establish a new presentation for the mirabolic Hecke algebra $\mHR$ over an arbitrary commutative ring $R$. Furthermore, we develop a Schur--Weyl duality between $\mH$ and the quantum group $U_q(\mathfrak{gl}_r)$. This duality is shown to be compatible with the classical duality between the Iwahori--Hecke algebra $\mathcal{H}_n(q)$ and the quantum group $U_q(\mathfrak{gl}_{r+1})$ via a see-saw principle (see \eqref{eq:see-saw} below). Building on the duality, we establish the Frobenius character formula and the recursive Murnaghan--Nakayama rule for the irreducible characters of $\mH$.

\subsection{}
Let us describe our results in further detail. Extending the presentations for $\mH$ over $\mathbb{C}(q)$ found in \cite{Si} and \cite{Ro1}, we first introduce a new presentation for $\mHR$ over arbitrary commutative ring $R$ in terms of generators $T_1,\dots, T_{n-1}$ and $P_1,\dots, P_n$, where the elements $P_i$ are idempotents for $1 \leq i \leq n$. Using this presentation, we obtain a new basis for $\mHR$ consisting of \emph{standard elements} $\{T_{(A,B,w)}\}$, where $A, B \subseteq \{1, 2, \dots, n\}$ with $|A|=|B|=k$, and $w \in \mathfrak{S}_{n-k}$ fixes the first $k$ indices. We prove that every trace function on the mirabolic Hecke algebra over the ring $R={\bf A}=\mathbb{Z}[v,v^{-1}]$ is uniquely determined by its values on the certain elements $\widehat{T}^{(n)}_{\mu}$ indexed by partitions of $k$ for all $0 \leq k \leq n$ and furthermore this construction yields a basis  for the cocenter $\mathscr{H}_{n,{\bf A}}(q)/[\mathscr{H}_{n,{\bf A}}(q), \mathscr{H}_{n,{\bf A}}(q)]$  (see Theorem~\ref{thm:basistrace} and Corollary~\ref{cor:spacetrace}). Specialization at $R=\mathbb{C}(q)$, this leads to well-defined notions of class polynomials and a character table for the mirabolic Hecke algebra, paralleling the framework for Hecke algebras developed by Geck and Pfeiffer \cite{GP1}. 

Building on this presentation, we establish a Schur--Weyl duality between $\mH$ and the quantum group $U_q(\mathfrak{gl}_r)$. Our construction is inspired by the framework developed by Halverson \cite{Ha} in the study of $q$-rook algebras. Let $V_r$ and $V_{r+1}$ be the natural representations of $U_q(\mathfrak{gl}_r)$ and $U_q(\mathfrak{gl}_{r+1})$ over $\mathbb{C}(q)$, respectively. The tensor space $V_{r+1}^{\otimes n}$ naturally admits a $U_q(\mathfrak{gl}_{r+1})$-module structure; by viewing $U_q(\mathfrak{gl}_r)$ as a subalgebra of $U_q(\mathfrak{gl}_{r+1})$, we obtain an action of $U_q(\mathfrak{gl}_r)$ on $V_{r+1}^{\otimes n}$. We construct explicit actions of the generators $T_1,T_2,\ldots,T_{n-1}$ and $P_1,P_2,\ldots,P_n$ on $V_{r+1}^{\otimes n}$ and prove that the actions of $U_q(\mathfrak{gl}_r)$ and $\mH$ satisfy the double centralizer property. This duality generalizes the classical Jimbo duality between $U_q(\mathfrak{gl}_r)$ and $\mathcal{H}_n(q)$ on $V_r^{\otimes n}$ and fits into the following commutative diagram satisfying the see-saw principle:
\begin{equation}\label{eq:see-saw}
\begin{array}{ccccc}
&U_q(\mathfrak{gl}_{r+1})\xrightarrow{\hspace{.5cm} \phi \hspace{.5cm}} &\End(V_{r+1}^{\otimes n} )&\xleftarrow{\hspace{.5cm} \psi \hspace{.5cm}} &\mathcal{H}_n(q)\\
&\rotatebox[origin=c]{90}{$\subset$} \qquad \qquad &\rotatebox[origin=c]{90}{$=$}  & &\rotatebox[origin=c]{-90}{$\subset$}\quad \\
 &U_q(\mathfrak{gl}_r)\xrightarrow{\hspace{.5cm} \Phi \hspace{.5cm}} &\End(V_{r+1}^{\otimes n} )&\xleftarrow{\hspace{.5cm} \Psi \hspace{.5cm}} &\mH\\
\end{array}
\end{equation}
where the first row represents the duality established in \cite{Ji}. This duality should be compared to the distinct geometric Schur--Weyl duality between mirabolic quantum groups and mirabolic Hecke algebras in \cite{FZM, Ro2}. We emphasize that while Rosso \cite{Ro1} established an abstract isomorphism between $\mH$ and the $q$-rook algebra $\mathscr{R}_n(q)$ via semisimplicity and dimension counting, a naive attempt to derive our duality by composing Rosso's isomorphism with Halverson's duality \cite{Ha} is hindered by the fact that Rosso's isomorphism is purely abstract; it remains a highly non-trivial task to define an isomorphism between $\mH$ and $\mathscr{R}_n(q)$ at the level of generators and relations, see Remark \ref{rem:isomorphism}.

By exploiting the Schur--Weyl duality on $V_{r+1}^{\otimes n}$, we derive a Frobenius-type character formula for $\mH$ (Theorem~\ref{thm:Frobenius}). Our approach is inspired by the work of Ram \cite{Ra}, who obtained a character formula for the type $A$ Hecke algebra via Schur--Jimbo duality. The symmetric functions appearing in our formula are certain Hall--Littlewood functions, analogous to those encountered in \cite{Ra}. Finally, we utilize this Frobenius-type formula to provide a recursive rule for computing characters by removing  strips from a partition. This rule serves as a mirabolic analogue of the Murnaghan--Nakayama rule, generalizing the results for $\mathcal{H}_n(q)$ found in \cite{Ra}.

Another feature of this duality is the introduction of \emph{mirabolic $q$-Schur algebras}. These algebras provide a natural analogue to the classical $q$-Schur algebras, which act as a bridge between the quantum group $U_q(\mathfrak{gl}_r)$ and the Hecke algebra $\mathcal{H}_n(q)$. While our construction is comparable to the mirabolic quantum groups and Schur algebras investigated geometrically in \cite{Ro2, GR}, the two frameworks remain essentially distinct. In forthcoming works, we shall address the cellular structure of both $\mHR$ and the mirabolic $q$-Schur algebras, alongside the module representation theory of $\mH$. 

\subsection{}
The article is organized as follows. In Section~\ref{sec:new-pre}, we introduce a new presentation and construct a basis for the mirabolic Hecke algebra $\mHR$ over an arbitrary commutative ring $R$. As applications in the case $R={\bf A}=\mathbb{Z}[q,q^{-1}]$, we provide a basis for the cocenter $\mathscr{H}_{n,{\bf A}}(q)/[\mathscr{H}_{n,{\bf A}}(q),\mathscr{H}_{n,{\bf A}}(q)]$ , which facilitates the definition of the character table for $\mH$ over the field $\mathbb{C}(q)$. In Section~\ref{sec:duality}, we develop a Schur--Weyl duality between $\mH$ and the quantum group $U_q(\mathfrak{gl}_r)$ acting on the tensor space $V_{r+1}^{\otimes n}$. Finally, we derive a Frobenius character formula in Section~\ref{sec:Frobenius} and a combinatorial rule for computing the irreducible characters of $\mH$ in Section~\ref{sec:MN}.

{\bf Acknowledgements.}  The author would like to thank Zhaobing Fan and Zhicheng Zhang for enlightening conversations and would also thank Halverson and Rosso for their work on $q$-rook algebras and mirabolic Hecke algebras over $\mathbb{C}(q)$ which helped lead to the presentation in  Theorem \ref{thm:new-pre-R}. 

\section{The mirabolic Hecke algebra $\mHR$}\label{sec:new-pre}
\subsection{The mirabolic Hecke algebra $\mHR$.} Let $R$ be a commutative ring and assume $q\in R$ is invertible. Following \cite{So1, So2}, we define the mirabolic Hecke algebra $\mHR$ as follows. 
\begin{definition} The mirabolic Hecke algebra $\mHR$ is defined to be the $R$-algebra generated by 
$T_0,T_1,\ldots, T_{n-1}$ subjection to the following relations:
\begin{align}
T_0^2&=(q-2)T_0+(q-1),\label{eq:mH-1}\\
  T_i^2&=(q-1)T_i+q,  \quad 1\leq i\leq n-1,\label{eq:mH-2}\\
T_iT_j&=T_jT_i,  \quad 0\leq i,j\leq n-1, |i-j|>1,\label{eq:mH-3}\\
T_iT_{i+1}T_i &=T_{i+1}T_iT_{i+1},   \quad 1\leq i\leq n-2,\label{eq:mH-4}\\
T_0T_1T_0T_1&=(q-1)(T_1T_0T_1+T_1T_0)-T_0T_1T_0,  \label{eq:mH-5}\\
T_1T_0T_1T_0&=(q-1) (T_1T_0T_1+T_0T_1)-T_0T_1T_0. \label{eq:mH-6}
\end{align}
\end{definition}
In the case $R=\mathbb{C}(q)$, we shall write $\mH=\mathscr{H}_{n,\mathbb{C}(q)}(q)$. 
For each $w\in\mathfrak{S}_n=\langle s_1,\ldots,s_{n-1}\rangle$, one can associate an element $T_w\in\mHR$. Let $\Omega_n$ be the Weyl group of type $B_n$, that is, $\Omega_n=C_2^n\rtimes\mathfrak{S}_n$ is the semidirect product of $C_2^n$ and $\mathfrak{S}_n$ with $C_2=\{1,z\}$ being the cyclic group of order $2$. Let $\mathsf{r}_1=(z,1,\ldots,1)\in C_2^n\subset C_2^n\rtimes\mathfrak{S}_n$ and $\mathsf{r}_i=s_{i-1}\mathsf{r}_{i-1}s_{i-1}$ for $2\leq i\leq n$. It is known that $\Omega_n$ is generated by $s_0:=\mathsf{r}_1,s_1,\ldots,s_{n-1}$ and each element $\omega\in\Omega_n$ can be uniquely written as $\omega=\mathsf{r}_1^{\alpha_1}\mathsf{r}_2^{\alpha_2}\cdots \mathsf{r}_n^{\alpha_n}w$ with $w\in\mathfrak{S}_n$ and $\alpha=(\alpha_1,\alpha_2,\ldots,\alpha_n)\in\mathbb{Z}_2^n$.  Accordingly
 set $M(\omega)=\{i\mid \alpha_i=1, 1\leq i\leq n\}\subseteq\{1,2,\ldots,n\}$ and $\omega^*=w$ for each $\omega=\mathsf{r}_1^{\alpha_1}\mathsf{r}_2^{\alpha_2}\cdots \mathsf{r}_n^{\alpha_n}w\in\Omega_n$.  Following \cite[Definition 2.23]{Si}, let $\Lambda_n$ be the subset of $\Omega_n$ defined via 
\begin{equation}\label{eq:Lambda}
\Lambda_n=\big\{\omega\in\Omega_n\mid \text{ if }i,j\in M(\omega) \text{ and }i<j, \text{then }(\omega^*)^{-1}(i)>(\omega^*)^{-1}(j)\big\}. 
\end{equation}
As $(\Omega_n, S)$ with $S=\{s_0,s_1,\ldots,s_{n-1}\}$ is a Coxeter system of type $B_n$,   each $\omega\in\Lambda_n\subset \Omega_n$ admits a reduced expression $\omega=s_{i_1}s_{i_2}\cdots s_{i_k}$ for $0\leq i_1,i_2,\ldots,i_k\leq n-1$ and $k\geq 0$ and then accordingly we set $T_{\omega}=T_{i_1}T_{i_2}\cdots T_{i_k}$. 
\begin{lemma}\cite[Proposition 2.25]{Si}\cite[Theorem 3.15, Proposition 6.7]{So1} \label{lem:old-basis}
For each $\omega\in\Lambda_n$, the element $T_\omega$ only depends on $\omega$. The set $\{T_{\omega}\mid \omega\in\Lambda_n\}$ is a $\mathbb{C}(q)$-basis of $\mH$ and moreover 
$$
\dim \mH=\sum_{k=0}^n \binom{n}{k}^2 k!.
$$
\end{lemma}
Clearly by \eqref{eq:Lambda} we have $\mathfrak{S}_n\subset \Lambda_n$ and hence $\{T_w\mid w\in\mathfrak{S}_n\}$ is linearly independent in $\mH$. This implies that the subalgebra of $\mH$ generated by $T_1,T_2,\ldots, T_{n-1}$ is isomorphic to the Hecke algebra $\mathcal{H}_n(q)$ associated to the symmetric group $\mathfrak{S}_n$.  From now on, we shall regard $\mathcal{H}_n(q)$ as a subalgebra of $\mH$.  

It is known that $\mH$ is semisimple over $\mathbb{C}(q)$ and its irreducible representations have been constructed by Siegel in \cite{Si} (see \cite[Proposition 4.15]{Ro1} for an equivalent construction) as follows. For each $n\geq 0$, let $\mathcal{P}_n$ be the set of partitions of $n$. 
\begin{lemma}\cite[Proposition 3.25]{Si}\label{lem:Si-irred}
 For each
$(\lambda,k)$ with $\lambda\in\mathcal{P}_k$ and $0\leq k\leq n$, there exists an irreducible representation
$N_{(\la,k)}^{(n)}$ of $\mH$ with character $\chi_{(\la,k)}^{(n)}$ such that
$\{N_{(\la,k)}^{(n)}\mid \la\in\mc{P}_k, 0\leq k\leq n\}$ forms a complete set of
nonisomorphic irreducible $\mH$-modules. 
\end{lemma}
\subsection{A new presentation for $\mHR$.} In \cite{Ro1}, Rosso provided an alternative presentation for $\mH$ over $\mathbb{C}(q)$ by showing that $\mH$ is a quotient algebra of the cyclotomic Hecke algebra with certain parameters, see \cite[Section 4]{Ro1} for details. Inspired by Harvelson's study on $q$-rook algebras in \cite{Ha}, we shall introduce another new presentation for $\mHR$ over an arbitrary commutative ring $R$  in the following. Due to \eqref{eq:mH-2}, we have $T_i^{-1}=q^{-1}(T_i-(q-1))$ for $1\leq i\leq n$. Let 
\begin{align}
P_1=&1-q^{-1}(T_0+1),\label{eq:P1}\\
 P_i=&-P_{i-1}T_{i-1}^{-1}P_{i-1}=-q^{-1}\big(P_{i-1}T_{i-1}P_{i-1}-(q-1)P^2_{i-1}\big),\quad 2\leq i\leq n \label{eq:Pi}. 
\end{align}
Clearly 
\begin{equation}\label{eq:Pi-recursive}
P_{i-1}T_{i-1}P_{i-1}=(q-1)P_{i-1}^2-qP_i
\end{equation}
\begin{lemma} \label{lem:new-rel}
The following relations hold in $\mHR$:
\begin{align}
P_i^2&=P_i,\quad 1\leq i\leq n,\label{eq:new-pre-1}\\
P_iP_j&=P_jP_i=P_i,\quad \forall 1\leq j<i\leq n,\label{eq:new-pre-2}\\
P_iT_j&=T_jP_i, \quad \forall 1\leq i<j\leq n,\label{eq:new-pre-3}\\
P_iT_j&=T_jP_i=-P_i,\quad \forall 1\leq j< i\leq n.\label{eq:new-pre-4}
\end{align}
\end{lemma}
\begin{proof}
Observe that \eqref{eq:new-pre-2} follows from \eqref{eq:Pi} and \eqref{eq:new-pre-1}. In addition, by \eqref{eq:mH-3} and \eqref{eq:Pi} we have $P_1T_j=T_jP_1$ for $j>1$ and then by induction on $i$ one can easily show \eqref{eq:new-pre-3} holds. Thus, it remains to show \eqref{eq:new-pre-1}, \eqref{eq:new-pre-2} and \eqref{eq:new-pre-4} hold and we will proceed by induction on $i$. 
By \eqref{eq:P1} we have $T_0=q(1-P_1)-1$ and hence $(T_0-(q-1))(T_0+1)=-q^2(1-P_1)P_1$. This means the relation \eqref{eq:mH-1} leads to 
\begin{equation}\label{eq:P1-idemp}
P_1^2=P_1
\end{equation}
and hence 
\begin{equation}\label{eq:P1P2}
P_2P_1=-P_1T_1^{-1}P_1^2=-P_1T_1^{-1}P_1=P_2=-P^2_1T_1^{-1}P_1=P_1P_2. 
\end{equation}
Meanwhile by \eqref{eq:mH-2} we have 
$$
(1+T_1)(1-q^{-1}T_1(1-P_1)T_1)=(1-q)+(1-q)T_1+P_1T_1+T_1P_1T_1
$$
and hence by \eqref{eq:P1-idemp}  and \eqref{eq:Pi-recursive}
\begin{align}
&P_1(1+T_1)(1-q^{-1}T_1(1-P_1)T_1)P_1\notag\\
&=(1-q)P_1+(2-q)P_1T_1P_1+P_1T_1P_1T_1P_1\label{eq:P2-idemp-1}\\
&=(1-q)P_1+(2-q)P_1T_1P_1+(P_1T_1P_1)^2\notag\\
&=(1-q)P_1+(2-q)((q-1)P_1-qP_2)+((q-1)P_1-qP_2)^2=q^2(P_2^2-P_2)\label{eq:P2-idemp-2}
\end{align}
One the other hand, by \eqref{eq:Pi} the term on the right hand side of \eqref{eq:P2-idemp-1} satisfies 
\begin{align*}
&(1-q)P_1+(2-q)P_1T_1P_1+P_1T_1P_1T_1P_1\\
=&((1-q)+(2-q)((1-q^{-1})T_1-q^{-1}T_0T_1))P_1\\
&+((1-q^{-1})T_1-q^{-1}T_0T_1)^2P_1\\
=&((1-q)+(2-q)((1-q^{-1})T_1-q^{-1}T_0T_1))P_1\\
&+\big((1-q^{-1})^2T_1^2-q^{-1}(1-q^{-1})T_1T_0T_1-q^{-1}(1-q^{-1})T_0T_1^2+q^{-2}T_0T_1T_0T_1\big)P_1\\
=&((1-q)+(2-q)((1-q^{-1})T_1-q^{-1}T_0T_1))P_1\\
&+\big((1-q^{-1})^2T_1^2-q^{-1}(1-q^{-1})T_0T_1^2+q^{-2}(q-1)T_1T_0-q^{-2}T_0T_1T_0\big)P_1\\
=&(1-q)P_1+(2-q)(1-q^{-1})T_1P_1-(2-q)q^{-1}T_0T_1P_1\\
&+(1-q^{-1})^2((q-1)T_1+q)P_1-q^{-1}(1-q^{-1})(q-1)T_0T_1P_1+q^{-1}(1-q^{-1})qP_1\\
&-q^{-2}(q-1)T_1P_1+q^{-2}T_0T_1P_1
=0
\end{align*}
where the third equality is due to \eqref{eq:mH-5} and the fourth equality is due to the fact $T_0T_1^2=(q-1)T_0T_1+qT_0$ and $T_0P_1=-P_1$ since $(T_0+1)P_1=0$ by \eqref{eq:mH-1}. This together with \eqref{eq:P2-idemp-2} means 
\begin{equation}\label{eq:P2-idemp-3}
P_2^2=P_2. 
\end{equation}
To prove \eqref{eq:new-pre-4} holds in the case $i=2,j=1$, we shall first show $P_2T_1=T_1P_2$.  Observe that by \eqref{eq:Pi} we have 
\begin{align*}
P_2T_1=-q^{-1}P_1T_1P_1T_1+q^{-1}(q-1)P_1T_1, \quad T_1P_2=-q^{-1}T_1P_1T_1P_1+q^{-1}(q-1)T_1P_1. 
\end{align*}
This together with \eqref{eq:P1} means 
\begin{equation*}
\begin{aligned}
P_2T_1-T_1P_2=&-q^{-1}(P_1T_1P_1T_1-T_1P_1T_1P_1)+q^{-1}(q-1)(P_1T_1-T_1P_1)\\
=&-q^{-1}\Big(\big((1-q^{-1})T_1-q^{-1}T_0T_1\big)^2-\big((1-q^{-1})T_1-q^{-1}T_1T_0\big)^2\Big)\\
&+q^{-2}(q-1)(T_1T_0-T_0T_1)\\
=&-q^{-3}(T_0T_1T_0T_1-T_1T_0T_1T_0)-q^{-1}(1-q^{-1})^2(T_1T_0-T_0T_1)\\
&+q^{-2}(q-1)(T_1T_0-T_0T_1)\\
=&-q^{-3}(T_0T_1T_0T_1-T_1T_0T_1T_0-(q-1)T_1T_0+(q-1)T_0T_1)=0,
\end{aligned}
\end{equation*}
where the last equality is due to \eqref{eq:mH-5} and \eqref{eq:mH-6}.  Therefore we obtain 
\begin{equation}\label{eq:P2T1-commute}
P_2T_1=T_1P_2
\end{equation}
Then by \eqref{eq:Pi}, \eqref{eq:P1P2} and \eqref{eq:P2-idemp-3}  we have 
\begin{align*}
P_2^2=&-P_2 P_1T_1^{-1}P_1=-q^{-1}P_2(T_1-(q-1))P_1\\
=&-q^{-1}P_2T_1P_1-q^{-1}(q-1)P_2P_1=-q^{-1}T_1P_2P_1-q^{-1}(q-1)P_2P_1\\
=&-q^{-1}T_1P_2-q^{-1}(q-1)P_2=P_2,
\end{align*}
which implies 
\begin{equation}\label{eq:T1P2}
P_2T_1=-P_2=T_1P_2. 
\end{equation} 
Thus by \eqref{eq:P1-idemp}, \eqref{eq:P1P2}, \eqref{eq:P2-idemp-3} and \eqref{eq:T1P2}, the relations \eqref{eq:new-pre-1}, \eqref{eq:new-pre-2} and \eqref{eq:new-pre-4} hold in the case $i=1,2$. By induction on $i$, we assume the relations \eqref{eq:new-pre-1}, \eqref{eq:new-pre-2} and \eqref{eq:new-pre-4} hold for $i=k\geq 2$. Now for the situation $i=k+1$, we first compute $P_{k+1}^2$ as follows 
\begin{align*}
P_{k+1}^2=&q^{-2}(P_kT_kP_k-(q-1)P_k)^2\\
=&q^{-2}(P_kT_kP_kT_kP_k-2(q-1)P_kT_kP_k+(q-1)^2P_k)\\
=&q^{-2}P_kT_k(-q^{-1}(P_{k-1}T_{k-1}P_{k-1}-(q-1)P_{k-1}))T_kP_k\\
&-2q^{-2}(q-1)P_kT_kP_k+q^{-2}(q-1)^2P_k\\
=&-q^{-3}P_kP_{k-1}T_kT_{k-1}T_kP_{k-1}P_k+q^{-3}(q-1)P_kP_{k-1}T_k^2P_k\\
&-2q^{-2}(q-1)P_kT_kP_k+q^{-2}(q-1)^2P_k\\
=&-q^{-3}P_kT_{k-1}T_kT_{k-1}P_k+q^{-3}(q-1)P_k((q-1)T_k+q)P_k\\
&-2q^{-2}(q-1)P_kT_kP_k+q^{-2}(q-1)^2P_k\\
=&-q^{-3}P_kT_kP_k+q^{-3}(q-1)^2P_kT_kP_k+q^{-2}(q-1)P_k\\
&-2q^{-2}(q-1)P_kT_kP_k+q^{-2}(q-1)^2P_k\\
=& -q^{-1}P_kT_kP_k+q^{-1}(q-1)P_k=-q^{-1}(P_kT_kP_k-(q-1)P_k)=P_{k+1}, 
\end{align*}
where the fourth equality is due to $T_kP_{k-1}=P_{k-1}T_k$ by \eqref{eq:new-pre-3}, the fifth equality is due to $P_kP_{k-1}=P_k$ as \eqref{eq:new-pre-2} holds in the case $i=k$ by induction on $i$ and \eqref{eq:mH-3}, and the sixth equality is due to 
$P_kT_{k-1}=-P_k$ as \eqref{eq:new-pre-4} holds in the case $i=k$ by induction on $i$. This proves \eqref{eq:new-pre-1} in the case $i=k+1$. 

In addition, 
$P_{k+1}P_k=-P_kT_k^{-1}P_k^2=-P_kT_k^{-1}P_k=P_{k+1}$ since  \eqref{eq:new-pre-1} holds in the case  $i=k$ by induction on $i$ and $P_kP_{k+1}=-P^2_kT_k^{-1}P_k=-P_kT_k^{-1}P_k=P_{k+1}$.  For $1\leq j<k$ we have $P_{k+1}P_j=-P_kT_k^{-1}P_kP_j=-P_kT_k^{-1}P_jP_k=-P_kP_jT_k^{-1}P_k=-P_jP_kT_k^{-1}P_k=P_jP_{k+1}$ as \eqref{eq:new-pre-2} holds in the case $i=k$ by induction on $i$. Putting together we obtain $P_{k+1}P_j=P_jP_{k+1}=P_{k+1}$ for $1\leq j\leq k$ and hence \eqref{eq:new-pre-2} holds in the case $i=k+1$. 

Finally,  for $1\leq j<k$, we have $P_{k+1}T_j=-P_kT_k^{-1}P_k T_j=P_kT_k^{-1}P_k=-P_{k+1}, T_jP_{k+1}=-T_jP_kT_k^{-1}P_k=P_kT_k^{-1}P_k=-P_{k+1}$ as $P_kT_j=T_jP_k=-P_k$ since \eqref{eq:new-pre-4} holds in the case $i=k$ by induction. For $j=k$, similar to the above calculation for $P^2_{k+1}$ we have 
\begin{align*}
P_{k+1}T_k=&-q^{-1}(P_kT_kP_kT_k-(q-1)P_kT_k)\\
=&-q^{-1}P_kT_k(-q^{-1}(P_{k-1}T_{k-1}P_{k-1}-(q-1)P_{k-1}))T_k+q^{-1}(q-1)P_kT_k\\
=&q^{-2}P_kT_kT_{k-1}T_kP_{k-1}-q^{-2}(q-1)P_kT_k^2+q^{-1}(q-1)P_kT_k\\
=&q^{-2}P_kT_{k-1}T_kT_{k-1}P_{k-1}+q^{-2}(q-1)P_kT_k-q^{-1}(q-1)P_k\\
=&-q^{-2}P_kT_kT_{k-1}P_{k-1}+q^{-2}(q-1)P_kT_k-q^{-1}(q-1)P_k\\
=&-q^{-2}P_kT_k(P_{k-1}T_{k-1}P_{k-1}-(q-1)P_{k-1})-q^{-1}(q-1)P_k\\
=&q^{-1}P_kT_kP_k-q^{-1}(q-1)P_k=-P_{k+1}, 
\end{align*}
where the sixth equality is due to $P_kT_k=P_kP_{k-1}T_k=P_kT_kP_{k-1}$ as \eqref{eq:new-pre-2} and \eqref{eq:new-pre-3} hold in the case $i=k$ while the last equality is due to \eqref{eq:Pi} and $P_k^2=P_k$ as 
\eqref{eq:new-pre-1}  hold in the case $i=k$ by induction on $i$. Similarly, one can show $T_kP_{k+1}=-P_{k+1}$. Thus, \eqref{eq:new-pre-4} holds in the case $i=k+1$. This proves the lemma. 
\end{proof}
We shall apply Lemma \ref{lem:new-rel} to provide a new presentation as well as a new basis for $\mHR$. Let $\mathscr{M}_{n,R}(q)$ be the algebra generated by $P_1,\ldots,P_n, T_1,\ldots,T_{n-1}$ subject to the relations \eqref{eq:mH-2}-\eqref{eq:mH-4}, \eqref{eq:Pi} and \eqref{eq:new-pre-1}-\eqref{eq:new-pre-4}. Here we abuse the notation by using the same notations $T_i, P_j$ in both $\mHR$ and $\mathscr{M}_{n,R}(q)$. By Lemma \ref{lem:new-rel} there exists a surjective algebra homomorphism satisfying 
\begin{equation} \label{eq:map-Phi}
\begin{aligned}
\Theta_R: &\mathscr{M}_{n,R}(q)\rightarrow \mHR, \\
P_1& \mapsto 1-q^{-1}(T_0+1),\quad T_i\mapsto T_i, \quad \text{ for }1\leq i\leq n-1. 
\end{aligned}
\end{equation} 
and accordingly 
\begin{equation} \label{eq:map-Phi-1}
\mathscr{M}_{n,R}(q)/\ker(\Theta_R)\cong \mHR.
\end{equation} 
Analogous to \cite{Ha}, we introduce the following spanning set for $\mathscr{M}_{n,R}(q)$. For each $0\leq k\leq n$, denote by $\mathfrak{S}'_{n-k}$ the Young subgroup of $\mathfrak{S}_n$ corresponding to the composition $\mu=(1^k,n-k)=(1,\ldots, 1,n-k)$. For $1\leq i\leq j\leq n-1$, set $T_{j,i}=T_{j-1}T_{j-2}\cdots T_i$, where we use the convention $T_{i,i}=1$.  Then for each subset $A=\{1\leq a_1<a_2<\cdots<a_k\leq n\}\subseteq\{1,2,\ldots,n\}$, define 
$$
T_A:=T_{a_1,1}T_{a_2,2}\cdots T_{a_k,k}\in\mHR. 
$$
For each $0\leq k\leq n$, set 
$$
\Gamma_k=\{(A,B,w)\mid A, B\subseteq\{1,2,\ldots,n\}, |A|=|B|=k, w\in\mathfrak{S}'_{n-k}\}
$$
and let $\Gamma^{(n)}=\cup_{k=0}^n \Gamma_k$. For each $(A,B,w)\in\Gamma^{(n)}$, define 
\begin{equation}\label{eq:TABw}
T_{(A,B,w)}:=T_AP_kT_wT_{B}^{-1}. 
\end{equation}
\begin{lemma}\label{lem:new-basis}
The algebra $\mathscr{M}_{n,R}(q)$ is spanned by the set $\{T_{(A,B,w)}\mid (A,B,w)\in\Gamma^{(n)}\}$. 
\end{lemma}
\begin{proof}
Let $ \mathscr{M}'_{n,R}(q)$ be the $R$-submodule of $\mathscr{M}_{n,R}(q)$ spanned by the set $\{T_{(A,B,w)}\mid (A,B,w)\in\Gamma^{(n)}\}$. 
Since $\mathscr{M}_{n,R}(q)$ is generated by $P_1, T_1, \dots, T_{n-1}$, each element of $\mathscr{M}_{n,R}(q)$ can be expressed as a linear combination of monomials in these generators.  
Observe that $T_1, \dots, T_{n-1}, P_1 \in \mathscr{M}'_{n,R}(q)$, because $T_i = T_{(\varnothing,\varnothing,s_i)}$ for $1 \le i \le n-1$, and $P_1 = T_{(\{1\},\{1\},1)}$.  
Then by induction on the number of factors in monomials in $P_1, T_1, \dots, T_{n-1}$, it suffices to show that
\[
\mathscr{M}'_n(q) T_i \subseteq \mathscr{M}'_{n,R}(q) \quad (1 \le i \le n-1),\qquad 
\mathscr{M}'_n(q) P_1 \subseteq \mathscr{M}'_{n,R}(q),
\]
in order to prove $\mathscr{M}_{n,R}(q) \subseteq\mathscr{M}'_{n,R}(q)$. 

Using a strategy parallel to the proof of \cite[Theorem 2.1]{Ha}, one can verify $\mathscr{M}'_{n,R}(q) T_i \subseteq \mathscr{M}'_{n,R}(q)$ for $1 \le i \le n-1$.  
The computation of $T_B^{-1} T_i^{-1}$ treated in four cases in \cite{Ha} applies here as well; we only note that in the case $i, i+1 \in B$, the relation $P_k T_j^{-1} = -P_k$ (for $j < k$) from \eqref{eq:new-pre-4} replaces the relation ``$P_k T_j^{-1} = q^{-1} P_k$'' appearing in \cite[Theorem 2.1]{Ha}.  
Similarly, the method of computing $T_{(A,B,w)} P_1$ in the proof of \cite[Theorem 2.1]{Ha} can be adapted, using \eqref{eq:new-pre-1}--\eqref{eq:new-pre-4}, to show $\mathscr{M}'_{n,R}(q) P_1 \subseteq \mathscr{M}_{n,R}(q)$.  
To save space, we omit the details. In summary, we obtain $\mathscr{M}_{n,R}(q) \subseteq\mathscr{M}'_{n,R}(q)$ and hence $\mathscr{M}_{n,R}(q)=\mathscr{M}'_{n,R}(q)$. Then the lemma is proved. 
\end{proof}
\begin{proposition}\label{prop:new-basis-R}
The algebra $\mathscr{M}_{n,R}(q)$ is a free $R$-module and  the set $\{T_{(A,B,w)}\mid (A,B,w)\in\Gamma^{(n)}\}$ is a $R$-basis for $\mathscr{M}_{n,R}(q)$. 
\end{proposition}
\begin{proof}
By Lemma \ref{lem:new-basis}, it suffices to show that the set $\{T_{(A,B,w)}\mid (A,B,w)\in\Gamma^{(n)}\}$ is $R$-linearly independent. Firstly, in the case $R={\bf A}=\mathbb{Z}[q,q^{-1}]$ with $q$ be an indeterminate and by base change to the fraction field $\mathbb{C}(q)$ of $\mathbb{Z}[q,q^{-1}]$ we obtain that the algebra $\mathscr{M}_n(q)$ over $\mathbb{C}(q)$ is spanned by the set $\{T_{(A,B,w)}\mid (A ,B,w)\in\Gamma^{(n)}\}$ and hence $\dim\mathscr{M}_n(q)\leq |\Gamma^{(n)}|=\dim\mH$.  This together with \eqref{eq:map-Phi-1} leads to  $\ker(\Theta_{\mathbb{C}(q)})=0$ and hence $\dim \mH\leq \dim\mathscr{M}_n(q)=\dim \mH=\sharp\Gamma^{(n)}$. Thus the set $\{T_{(A,B,w)}\mid (A ,B,w)\in\Gamma^{(n)}\}$ is actually a basis for $\mathscr{M}_n(q)$ over $\mathbb{C}(q)$ and hence $\{T_{(A,B,w)}\mid (A ,B,w)\in\Gamma^{(n)}\}$ is linearly independent over ${\bf A}:=\mathbb{Z}[q,q^{-1}]$. Then $\{1\otimes T_{(A,B,w)}\mid (A ,B,w)\in\Gamma^{(n)}\}$ is a $R$-basis for the $R$-module $R\otimes_{{\bf A}}\mathscr{M}_{n,{\bf A}}(q)$, where we regard $R$ as ${\bf A}$-module by specializing $q$ to the invertible element still denoted by $q\in R$.  Then the canonical map $\mathscr{M}_{n,R}(q)\rightarrow R\otimes_{{\bf A}}\mathscr{M}_{n,{\bf A}}(q)$ sending $P_j, T_i$ to $1\otimes P_j, 1\otimes T_i$ for $1\leq i,\leq n-1,1\leq j\leq n$ maps the set $\{T_{(A,B,w)}\mid (A,B,w)\in\Gamma^{(n)}\}$ to a basis of $R\otimes_{\bf A}\mathscr{M}_{n,{\bf A}}(q)$ and hence it is linearly independent. This proves the proposition. 
\end{proof}
\begin{corollary}\label{cor:R-base}
Let $R$ be a commutative ring which is a ${\bf A}$-module via the map ${\bf A}\rightarrow R, q\mapsto q$. Then $\mathscr{M}_{n,R}(q)\cong R\otimes_{{\bf A}}\mathscr{M}_{n,{\bf A}}(q)$. 
\end{corollary}

\begin{theorem}\label{thm:new-pre-R} 
Suppose $R$ is a commutative ring with $q\in R$ invertible. Then $\mHR\cong \mathscr{M}_{n,R}(q)$. Hence $\mHR$ is the $R$-algebra generated by $P_1,\ldots,P_n, T_1,\ldots,T_{n-1}$ subject to the relations \eqref{eq:mH-2}-\eqref{eq:mH-4}, \eqref{eq:Pi} and \eqref{eq:new-pre-1}-\eqref{eq:new-pre-4} and moreover $\mHR$ is a free $R$-module with a basis $\{T_{(A,B,w)}\mid (A,B,w)\in\Gamma^{(n)}\}$. 
\end{theorem}
\begin{proof}
By Proposition \ref{prop:new-basis-R} and  \eqref{eq:map-Phi-1}, 
it suffices to prove that $\Theta_R$ is injective. 
By the proof of Proposition \ref{prop:new-basis-R} we obtain that $\ker{\Theta_{{\bf A}}}=0$ as $\ker{\Theta_{{\bf A}}}\subset \ker{\Theta_{\mathbb{C}(q)}}$. Then $\Theta_{\bf A}: \mathscr{M}_{n,{\bf A}}(q)\rightarrow \mathscr{H}_{n,{\bf A}}(q)$ is an isomorphism and hence there exists the inverse homomorphism $\Xi_{{\bf A}}: \mathscr{H}_{n,{\bf A}}(q)\rightarrow \mathscr{M}_{n,{\bf A}}(q)$ satisfying  
\begin{equation}\label{eq:Psi-A}
\Xi_{{\bf A}}(T_0)=q(1-P_1)-1, \quad \Xi_{{\bf A}}(T_i)=T_i, \quad 1\leq i\leq n-1
\end{equation}
Then $\Xi_{{\bf A}}$ induces a homomorphism $\Xi_R: \mathscr{H}_{n,R}(q)\rightarrow R\otimes_{\bf A}\mathscr{H}_{n,{\bf A}}(q)\rightarrow R\otimes_{\bf A}\mathscr{M}_{n,{\bf A}}(q)\cong \mathscr{M}_{n,R}(q)$ by Corollary \ref{cor:R-base} which satisfies $\Xi_R(T_0)=q(1-P_1)-1$ and  $\Xi_{R}(T_i)=T_i$ due to \eqref{eq:Psi-A}. Then clearly $\Xi_R\Theta_R=\text{id}$ and hence $\Theta_R$ is injective. This proves the assertion that $\Theta_R$ is an isomorphism. Then the theorem is proved. 
\end{proof}

\begin{remark}\label{rem:isomorphism}
Over the field $\mathbb{C}(q)$, the mirabolic Hecke algebra $\mH$ is semisimple. Rosso \cite{Ro1} established an abstract isomorphism between $\mH$ and the $q$-rook algebra $\mathscr{R}_n(q)$ by decomposing both into direct sums of matrix algebras and comparing the dimensions of their respective irreducible representations. Theorem \ref{thm:new-pre-R} provides a new presentation for the mirabolic Hecke algebra $\mathscr{H}_{n,R}(q)$ over an arbitrary commutative ring $R$. By comparing the relations \eqref{eq:Pi}--\eqref{eq:new-pre-4} with the presentation of $\mathscr{R}_n(q)$ given in \cite[(2.1)(A4)--(A7)]{Ha}, one observes that the two sets of relations are related by interchanging the parameters $q$ and $-1$. Within the Hecke algebra $\mathcal{H}_n(q)$---viewed as a subalgebra of both $\mH$ and $\mathscr{R}_n(q)$---this correspondence is realized via the Mullineux involution $T_i \mapsto -q^{-1}T_i^{-1}$ for $1 \leq i \leq n-1$. 
However, it appears that this involution cannot be extended to an isomorphism between $\mHR$ and $\mathscr{R}_{n,R}(q)$ that recovers Rosso's isomorphism over $\mathbb{C}(q)$. We note that by \cite[Theorem 4.12]{Ro1}, By \cite[Theorem 4.12]{Ro1}, $\mH$ is isomorphic to the quotient of cyclotomic Hecke algebra $\mathcal{H}_n(q;0,1)$ with the cyclotomic condition $X_1(X_1-1)=0$. Inspired by \cite{JL}, a potential remedy may involve a systematic change of parameters; we shall address this construction as well as the cellular structure of $\mathscr{H}_{n,R}(q)$ in a forthcoming paper.
\end{remark}

\subsection{The space of trace functions and character table of $\mH$.} Let $R$ be a commutative ring with $q\in R$ invertible. For a
$R$-algebra $\mc{H}$ which is a free $R$-module, a {\em trace
function} on $\mc{H}$ is an $R$-linear map $\phi: \mc{H}\rightarrow
R$ satisfying
$$
\phi(hh')=\phi(h'h) \text{ for } h, h'\in\mc{H}. 
$$
For $h, h'\in\mc H$, define their commutator by $[h,h']=hh'-h'h$,
and let $[\mc H,\mc H]$ be the $R$-submodule of $\mc H$ spanned by
all commutators (not super-commutators!). Observe that a linear map
$\phi:\mc H\rightarrow R$ is a trace function if and only if ${[\mc H,\mc H]}\subseteq {\rm Ker}\phi$.
Thus, the space of trace functions on $\mc H$ is canonically
isomorphic to the dual space $\Hom_{R}({(\mc H/[\mc H,\mc H])},
R)$ of ${\mc H}/{[\mc H,\mc H]}$.

Recall ${\bf A}=\mathbb{Z}[q,q^{-1}]\subseteq\C(q)$, and  $\mH$ is $\bf A$-free and
$\mH=\C(q)\otimes_{\bf A}\mathscr{H}_{n,\bf A}(q)$. For a composition $\ga=(\ga_1,\ldots,\ga_{\ell})$ of $n$, let
\begin{align}
T_{{\ga,j}}&= T_{\ga_1+\ldots+\ga_{j-1}+1}
T_{\ga_1+\ldots+\ga_{j-1}+2} \cdots T_{\ga_1+\cdots+\ga_{j-1}+\ga_j-1},
 \quad 1\leq j\leq \ell,
   \notag \\
T_{w_{\ga}}&=T_{{\ga,1}}T_{{\ga,2}}\cdots T_{{\ga,{\ell}}}.
  \label{eq:Twga}
\end{align}
Equivalently, $T_{w_{\ga}}$ is the element corresponding to the permutation $w_\ga\in\mathfrak{S}_n$ with 
\begin{equation}  \label{eq:wha}
w_{\ga}=(1,\ldots,\ga_1)(\ga_1+1,\ldots, \ga_1+\ga_2)
\cdots(\ga_1+\cdots+\ga_{\ell-1}+1,\ldots, \ga_1+\cdots+\ga_{\ell}).
\end{equation}
If $\mu=(\mu_1,\ldots,\mu_\ell)$ is a composition with $0\leq |\mu|=k\leq n$, then ${\mu}^{\uparrow n}=(\underbrace{1,\ldots, 1}_{n-k}, \mu_1,\ldots,\mu_\ell)$ is a composition of $n$ 
and accordingly we have 
\begin{align*}
w_{\mu^{\uparrow n}} &= (1)(2)\cdots(n-k) \cdot (n-k+1,\ldots,n-k+\mu_1)(n-k+\mu_1+1,\ldots,n-k+\mu_1+\mu_2) \cdot \notag \\
&\quad \cdots (n-k+\mu_1+\cdots+\mu_{\ell-1}+1,\ldots, n-k+\mu_1+\cdots\mu_{\ell-1}+\mu_\ell), \notag \\
\end{align*}
and according
\begin{align}
T_{w_{\mu^{\uparrow n}}} &= T_{\mu^{\uparrow n}, 1}T_{\mu^{\uparrow n}, 2}\cdots T_{\mu^{\uparrow n}, n-k} \cdot T_{\mu^{\uparrow n}, n-k+1}T_{\mu^{\uparrow n},n-k+2}\cdots T_{\mu^{\uparrow n},n-k+\ell} \notag \\
&= T_{\mu^{\uparrow n}, n-k+1}T_{\mu^{\uparrow n},n-k+2}\cdots T_{\mu^{\uparrow n},n-k+\ell} \label{eq:w-mu-T}
\end{align}

as $T_{{\mu^{\uparrow n}}, a}=1$ for $1\leq a\leq n-k$. 
Set 
\begin{equation}\label{eq:hat-T-mu}
\widehat{T}^{(n)}_\mu=P_{n-|\mu|}T_{w_{{\mu^{\uparrow n}}}}=P_{n-|\mu|}T_{{\mu^{\uparrow n}}, n-k+1}T_{{\mu^{\uparrow n}},n-k+2}\cdots T_{{\mu^{\uparrow n}},n-k+\ell}, 
\end{equation}
where we use the convention $P_0=1$. 
\begin{lemma}\label{lem:T-mu-la}
Suppose $\mu$ is a composition with $0\leq |\mu|\leq n$. Let $\la$ be the partition obtained by rearrangement of the parts of $\mu$. Then 
$$
\widehat{T}^{(n)}_{\mu}\equiv\widehat{T}^{(n)}_\la\mod [\mathscr{H}_{n,\bf A}(q), \mathscr{H}_{n,\bf A}(q)]. 
$$
\end{lemma}
\begin{proof}
Let $|\mu|=k=|\la|$. 
By \eqref{eq:w-mu-T} we observe that both ${T}_{w_{{\mu^{\uparrow n}}}}$ and ${T}_{w_{{\la}^{\uparrow n}}}$ belong to the subalgebra $\mathcal{H}'_{k,{\bf A}}(q)$ of $\mathscr{H}_{n,{\bf A}}(q)$ generated by $T_{n-k+1},T_{n-k+2},\ldots, T_{n-1}$. As $\mathcal{H}'_{k,{\bf A}}(q)\cong \mathcal{H}_{k,{\bf A}}(q)$, it is known (see \cite[Theorem 5.1]{Ra} or \cite[Section 8.2]{GP2}) that 
${T}_{w_{{\mu^{\uparrow n}}}}\equiv {T}_{w_{{\la}^{\uparrow n}}}\mod [\mathcal{H}'_{k,{\bf A}}(q),\mathcal{H}'_{k,{\bf A}}(q)]$. This means there exists $z\in [\mathcal{H}'_{k,{\bf A}}(q),\mathcal{H}'_{k,{\bf A}}(q)]$ 
such that ${T}_{w_{{\mu^{\uparrow n}}}}={T}_{w_{{\la}^{\uparrow n}}}+z$. Then 
\begin{equation}\label{eq:T-mu-la}
P_{n-k}{T}_{w_{{\mu^{\uparrow n}}}}=P_{n-k}{T}_{w_{{\la}^{\uparrow n}}}+P_{n-k}z
\end{equation} 
On the other hand, by \eqref{eq:new-pre-3} observe that $P_{n-k}$ commutes with $T_{n-k+1},T_{n-k+2},\ldots, T_{n-1}$ and hence $P_{n-k}$ commutes with every element in $\mathcal{H}'_{k,{\bf A}}(q)$. Then 
for $z_1,z_2\in \mathcal{H}'_{k,{\bf A}}(q)$ we obtain $P_{n-k}[z_1,z_2]=P_{n-k}(z_1z_2-z_2z_1)=P_{n-k}z_1z_2-z_2P_{n-k}z_1=[P_{n-k}z_1,z_2]$ and hence $P_{n-k}[z_1,z_2]\in [\mathscr{H}_{n,{\bf A}}(q), \mathscr{H}_{n,{\bf A}}(q)]$. This together with \eqref{eq:T-mu-la} leads to $P_{n-k}{T}_{w_{{\mu^{\uparrow n}}}}-P_{n-k}{T}_{w_{{\la}^{\uparrow n}}}=P_{n-k}z\in [\mathscr{H}_{n,{\bf A}}(q), \mathscr{H}_{n,{\bf A}}(q)]$. 
Then by \eqref{eq:hat-T-mu} the lemma is proved. 
\end{proof}
Let $\mathscr{H}'_{n-1,{\bf A}}(q)$ be the subalgebra of $\mathscr{H}_{n,{\bf A}}(q)$ generated by $P_2,\ldots, P_n$ and $T_2,\ldots, T_{n-1}$ and let $\iota: \mathscr{H}_{n-1,{\bf A}}(q)\longrightarrow\mathscr{H}'_{n-1,{\bf A}}(q)$ be the natural isomorphism: 
\begin{equation}\label{eq:iota}
\begin{aligned}
\iota: \mathscr{H}_{n-1,{\bf A}}(q)&\longrightarrow\mathscr{H}'_{n-1,{\bf A}}(q) \\
T_i&\mapsto T_{i+1}, P_j\mapsto P_{j+1}, \quad 1\leq i\leq n-2, 1\leq j\leq n-1
\end{aligned}
\end{equation}
Then by \eqref{eq:w-mu-T}  and \eqref{eq:TABw} it is straightforward to verify
\begin{equation}\label{eq:iota-Tw}
\iota(T_{w_{\mu^{\uparrow n-1}}})=T_{w_{\mu^{\uparrow n}}}. 
\end{equation}
as well as 
\begin{equation}\label{eq:iota-image}
\iota(T_{(A,B,w)})=
\left\{
\begin{array}{ll}
T_{w'},&\text{ if }|A|=|B|=0,\\
T_{(A',B',w')}, &\text{ if }1\leq |A|=|B|\leq n-1
\end{array}
\right.
\end{equation}
where $A'=\{1\leq a'_1<a'_2<\cdots<a'_{k+1}\leq n\},B'=\{1\leq b'_1<b'_2<\cdots<b'_{k+1}\leq n\}, w'=s_{i_1+1}\cdots s_{i_t+1}$ with $a'_1=1,a'_j=a_j+1, b'_1=1,b'_j=b_j+1$ if $A=\{1\leq a_1<a_2<\cdots a_k\leq n-1\}$, 
$B=\{1\leq b_1<b_2<\cdots b_k\leq n-1\}$ and $w=s_{i_1}s_{i_2}\cdots s_{i_t}\in \mathfrak{S}_{n-1}$. 
\begin{lemma}\label{lem:rho-A}
(1)  $P_1 \mathscr{H}_{n,{\bf A}}(q)P_1=P_1\mathscr{H}'_{n-1,{\bf A}}(q)$ and moreover $P_1\mathscr{H}'_{n-1,{\bf A}}(q)$ is a free $R$-module with basis $\{P_1\iota(T_{(A,B,w)})\mid (A,B,w)\in\Gamma^{(n-1)}\}$.\\

(2)  The map $\rho_{\bf A}: \mathscr{H}_{n-1,{\bf A}}(q)\rightarrow P_1 \mathscr{H}_{n,{\bf A}}(q)P_1$ defined via 
\begin{equation}\label{eq:rho-A}
\rho_{\bf A}(T_i)=P_1T_{i+1}=T_{i+1}P_1,\quad \rho_{\bf A}(P_j)=P_{j+1}, \text{ for }1\leq i\leq n-2, 1\leq j\leq n-1. 
\end{equation}
is an isomorphism. 

\end{lemma}
\begin{proof}
By a strategy parallel to the proof of \cite[Lemma 5.3]{DHP}, it is straightforward to verify the first assertion of part (1). Clearly $P_1\mathscr{H}'_{n-1,{\bf A}}(q)$ is a $R$-module spanned by the set $\{P_1\iota(T_{(A,B,w)})\mid (A,B,w)\in\Gamma^{(n-1)}\}$. Meanwhile by \eqref{eq:iota-image} we have 
\begin{equation*}
\begin{aligned}
P_1\iota(T_{(A,B,w)})=&P_1T_{(A',B',w')}\\
=&\left\{
\begin{array}{ll}
P_1T_{w'}=T_{(\{1\},\{1\}, w')},&\text{ if } |A|=|B|=0,\\
P_1T_{A'}P_{k+1}T_{w'}T^{-1}_{B'}=T_{A'}P_1P_{k+1}T_{w'}T^{-1}_{B'}&\\
=T_{A'}P_{k+1}T_{w'}T^{-1}_{B'}=T_{(A',B',w)},&\text{ if }|A|=|B|=k>0
\end{array}
\right.
\end{aligned}
\end{equation*}
for each $(A,B,w)\in\Gamma^{(n-1)}$. This implies the set $\{P_1\iota(T_{(A,B,w)})\mid (A,B,w)\in\Gamma^{(n-1)}\}$ is $R$-linearly independent by Theorem \ref{thm:new-pre-R}.  This proves the second assertion of (1). 

By the first assertion of (1), it is obvious that  $\rho_{\bf A}$ is a surjective homomorphism. In addition, by \eqref{eq:rho-A} and \eqref{eq:iota-image} we observe that 
$$
\rho_{\bf A}(T_{(A,B,w)})=P_1\iota(T_{(A,B,w)}) \,\quad \text{ for }(A,B,w)\in\Gamma^{(n-1)}
$$
Then by the second assertion of (1), the homomorphism $\rho_{\bf A}$ sends a basis to a basis and hence it is injective. 
This proves the lemma. 
\end{proof}
\begin{theorem}\label{thm:classpoly}
For each $(A,B,w)\in \Gamma^{(n)}$, there exists $f_{(A,B,w)}^\la\in {\bf A}$ such that 
\begin{equation}\label{eq:commutator}
T_{(A,B,w)}\equiv \sum_{k=0}^n\sum_{\la\in\mc{P}_k}f^\la_{(A,B,w)}\widehat{T}^{(n)}_\la \mod [\mathscr{H}_{n,\bf A}(q), \mathscr{H}_{n,\bf A}(q)]. 
\end{equation}
\end{theorem}
\begin{proof}
We shall prove the theorem by induction $n$. In the case $n=1$, the algebra $\mathscr{H}_{1,{\bf A}}(q)$ admits the basis $\{T_{(\emptyset,\emptyset, 1)},T_{(\{1\},\{1\},1)} \}$ and moreover $T_{(\{1\},\{1\},1)}=P_1=\widehat{T}^{(1)}_\mu$ with the composition $\mu=(0)$ and $T_{(\emptyset,\emptyset, 1)}=1=\widehat{T}^{(1)}_{\ga}$ with the composition $\ga=(1)$. So the theorem holds in the case $n=1$. 

In the case $n=2$, the algebra $\mathscr{H}_{2,{\bf A}}(q)$ admits the basis consisting of the following elements 
\begin{align}
T_{(\emptyset,\emptyset, 1)}=1, \quad T_{(\emptyset,\emptyset, s_1)}=T_1,\quad T_{(\{1\},\{1\}, 1)}=P_1,\quad T_{(\{1\},\{2\}, 1)}=P_1T_1^{-1},\notag\\
T_{(\{2\},\{1\}, 1)}=T_1P_1, \quad T_{(\{2\},\{2\}, 1)}=T_1P_1T_1^{-1},\quad T_{(\{1,2\},\{1,2\}, 1)}=T_1P_2T_1^{-1}. \notag
\end{align}
A direct calculation using \eqref{eq:hat-T-mu} and \eqref{eq:w-mu-T} gives rise to
$$
1=\widehat{T}^{(2)}_{\la}, \quad T_1=\widehat{T}^{(2)}_\mu,\quad P_1=\widehat{T}^{(2)}_\nu,\quad P_2=\widehat{T}^{(2)}_\gamma
$$
corresponding the four partitions $\la=(1,1), \mu=(2), \nu=(1), \gamma=(0)$ belonging to the set $\cup_{k=0}^2\mathcal{P}_k$. 
In addition, $P_1T_1^{-1}=P_1^2T_1^{-1}\equiv P_1T_1^{-1}P_1\equiv -P_2\equiv -\widehat{T}^{(2)}_{\gamma} \mod [\mathscr{H}_{2,\bf A}(q), \mathscr{H}_{2,\bf A}(q)]$ and similarly $T_1P_1=T_1P_1^2\equiv P_1T_1P_1\equiv (q-1)P_1-qP_2\equiv (q-1) \widehat{T}^{(2)}_{\nu}-q\widehat{T}^{(2)}_{\gamma} \mod [\mathscr{H}_{2,\bf A}(q), \mathscr{H}_{2,\bf A}(q)]$ by \eqref{eq:Pi-recursive}. Finally $T_1P_1T_1^{-1}\equiv P_1 \equiv \widehat{T}^{(2)}_{\nu}\mod [\mathscr{H}_{2,\bf A}(q), \mathscr{H}_{2,\bf A}(q)]$ and $T_1P_2T_1^{-1}\equiv P_2\equiv \widehat{T}^{(2)}_{\gamma}\mod [\mathscr{H}_{2,\bf A}(q), \mathscr{H}_{2,\bf A}(q)]$. Putting together we obtain that the theorem holds in the case $n=2$. 

Now assume the theorem holds for $n\leq m$. We shall show it also holds in the case $n=m+1$. By \eqref{eq:TABw}, if $|A|=|B|=0$, then $T_{(A,B,w)}=T_w$ with $w\in\mathfrak{S}_n$. Then it is known (see \cite[Theorem 5.1]{Ra} or \cite[Section 8.2]{GP2}) that there exists $f_w^\la\in{\bf A}$ such that $T_w\equiv \sum_{\la\in\mathcal{P}_n}f_w^\la T_{w_\la}\equiv \sum_{\la\in\mathcal{P}_n}f_w^\la \widehat{T}^{(n)}_{\la} \mod [\mathscr{H}_{n,\bf A}(q), \mathscr{H}_{n,\bf A}(q)]$ due to $\widehat{\la}=\la$ for $\la\in\mathcal{P}_n$ and \eqref{eq:hat-T-mu}. So the theorem holds in this case. Now assume $|A|=|B|=k>0$. Then 
\begin{align*}
T_{(A,B,w)}=&(T_AT_w)(P_kT_B^{-1})\equiv P_kT_B^{-1}T_AT_w \equiv P_1(P_1P_kT_B^{-1}T_AT_w)\\
&\equiv P_1P_kT_B^{-1}T_AT_wP_1\mod [\mathscr{H}_{n,\bf A}(q), \mathscr{H}_{n,\bf A}(q)] 
\end{align*}
since $P_1P_k=P_k$ due to \eqref{eq:new-pre-2}. Then by induction on $n$ as well Lemma \ref{lem:rho-A} (2) we obtain 
\begin{align}\label{eq:A-B-nonemp}
P_1P_kT_B^{-1}T_AT_wP_1\equiv \sum_{t=0}^{n-1}\sum_{\mu\in\mathcal{P}_t}g^\mu_{(A,B,w)}\rho_{\bf A}(\widehat{T}^{(n-1)}_\mu) \mod [\mathscr{H}_{n,\bf A}(q), \mathscr{H}_{n,\bf A}(q)]
\end{align}
for some $g^\mu_{(A,B,w)}\in {\bf A}$.  Clearly by \eqref{eq:hat-T-mu}, we have 
$$
\widehat{T}^{(n-1)}_\mu=P_{n-1-|\mu|}T_{w_{{\mu}^{\uparrow n-1}}}\in \mathscr{H}_{n-1,A}(q), 
$$
where we recall $P_0=1$. 
Then by  \eqref{eq:rho-A}, \eqref{eq:iota-image} and \eqref{eq:iota-Tw} we obtain 
\begin{equation}
\rho_{{\bf A}}(\widehat{T}^{(n-1)}_\mu)=P_1\iota(P_{n-1-|\mu|}T_{w_{{\mu}^{\uparrow n-1}}})=P_1P_{n-|\mu|}T_{w_{{\mu}^{\uparrow n}}}=P_{n-|\mu|}T_{w_{{\mu}^{\uparrow n}}}=\widehat{T}^{(n)}_{\mu}
\end{equation}
due to the fact $P_1P_{n-|\mu|}=P_{n-|\mu|}$ for each $\mu\in\mathcal{P}_k$ with $0\leq k\leq n-1$ by \eqref{eq:new-pre-2}. 
Putting together, the theorem is proved. 
\end{proof}

Theorem~\ref{thm:classpoly} has the following implication. 

\begin{theorem}\label{thm:basistrace}
$\mHA/[\mHA,\mHA]$ is a free $\bf A$-module, with a basis
consisting of the images of $\widehat{T}^{(n)}_{\la}$ for $\la\in\cup_{k=0}^n\mc{P}_k$
under the canonical projection $\mHA\rightarrow\mHA/[\mHA,\mHA]$.
\end{theorem}

\begin{proof}
By Theorem~\ref{thm:classpoly} $\mHA/[\mHA,\mHA]$ is
spanned by the images of $\widehat{T}^{(n)}_{\la}$ with
$\la\in\cup_{k=0}^n\mathcal{P}_k$ under the projection
$\mHA\rightarrow\mHA/[\mHA,\mHA]$. Passing to the fraction field
$\mathbb{C}(q)$ for mirabolic Hecke algebra $\mH=\mathscr{H}_{n,\mathbb{C}(q)}(q)$, the images of the elements
$\widehat{T}^{(n)}_{\la}$ with
$\la\in\cup_{k=0}^n\mathcal{P}_k$ remain to be a spanning set for
$\mH/[\mH,\mH]$. It is known that $\mH$ is
semisimple and its non-isomorphic irreducible characters are
parametrized by the set $\{(\la,k)\mid \la\in\mc{P}_k,0\leq k\leq n\}$ which has the same cardinality as $\cup_{k=0}^n\mathcal{P}_k$ by Lemma \ref{lem:Si-irred}. It follows that the dimension of the
space of trace functions on $\mH$ is
$$
\dim_{\mathbb{C}(q)} \mH/[\mH,\mH] = |\cup_{k=0}^n\mathcal{P}_k|. 
$$
This implies that the images of $\widehat{T}^{(n)}_{\la}$ with
$\la\in\cup_{k=0}^n\mathcal{P}_k$ are
linearly independent in the $\mH/[\mH,\mH] $ as well as in
$\mHA/[\mHA,\mHA]$. This proves the theorem.
\end{proof}

\begin{corollary}\label{cor:spacetrace}
Every trace function $\chi:\HCa\rightarrow {\bf A}$ is uniquely
determined by its values on the elements $\widehat{T}^{(n)}_{\la}$ with
$\la\in\cup_{k=0}^n\mathcal{P}_k$. Moreover, the polynomials $f_{(A,B,w)}^\la$ in
\eqref{eq:commutator} are uniquely determined by $(A,B,w)$ and
$\la$.
\end{corollary}

For $(A,B,w)\in\Gamma$ and
$\la\in \cup_{k=0}^n\mathcal{P}_k$. $f^\la_{(A,B,w)}$ are called {\em class
polynomials}, and they are mirabolic Hecke analogues of the class polynomials
introduced by Geck-Pfeiffer \cite[Definition 1.2(2)]{GP1} (cf.
\cite[Section 8.2]{GP2}) for Hecke algebras associated to finite
Weyl groups.
The square matrix
$$
\big[\chi^{(n)}_{\la,k}(\widehat{T}^{(n)}_{\nu}) \big]_{\la\in\mc{P}_k,0\leq k\leq n, \nu\in\cup_{t=0}^n\mc{P}_t}
$$
is called the {\em character table} of the mirabolic Hecke algebra
$\mH$. By Corollary~\ref{cor:spacetrace} and the linear
independence of irreducible characters $\chi^{(n)}_{(\la,k)}$ for
$\la\in\mc{P}_k$ and $0\leq k\leq n$, the square matrix $\big[\chi^{(n)}_{\la,k}(\widehat{T}^{(n)}_{\nu}) \big]_{\la\in\mc{P}_k,0\leq k\leq n, \nu\in\cup_{t=0}^n\mc{P}_t}$ is invertible in $\mathbb{C}(q)$.

\section{Schur-Weyl duality between $U_q(\mathfrak{gl}_r)$ and $\mH$}\label{sec:duality}
In this section, we construct a Schur--Weyl duality between \(U_q(\mathfrak{gl}_r)\) and \(\mH\) over the field \(\mathbb{C}(v)\) with $v^2=q$.  

Let $r\geq 1$. It is known that the quantum group $U_q(\mathfrak{gl}_{r+1})$ with Chevalley generators $\mathsf{K}^{\pm 1}_i, \mathsf{E}_j, \mathsf{F}_j, 1\leq i\leq r+1,1\leq i\leq r$ admits a natural representation $V_{r+1}=\mathbb{C}({v})$-span$\{u_1,u_2,\ldots,u_r,u_{r+1}\}$ satisfying 
$$
\mathsf{K}_i u_a=
\left\{
\begin{array}{cc}
{v} u_a,&\text{ if }a=i, \\
u_a,&\text{ if }a\neq i. 
\end{array}
\right.
\mathsf{E}_i u_a=
\left\{
\begin{array}{cc}
u_{a+1},&\text{ if }a=i, \\
0,&\text{ if }a\neq i. 
\end{array}
\right.
\mathsf{F}_i u_a=
\left\{
\begin{array}{cc}
u_{a-1},&\text{ if }a=i+1, \\
0,&\text{ if }a\neq i+1. 
\end{array}
\right.
$$
and moreover 
\begin{equation}\label{eq:restriction}
\text{res}^{U_q(\mathfrak{gl}_{r+1})}_{U_q(\mathfrak{gl}_r)}V_{r+1}=\mathbb{C}({v})\text{-span}\{u_1,u_2,\ldots,u_r\}\oplus\mathbb{C}({v})\text{-span}\{u_{r+1}\}\cong V_r\oplus V_0, 
\end{equation}
where $V_{r}$ is the natural representation of $U_q(\mathfrak{gl}_r)$ and $V_0$ is the trivial representation satisfying $\mathsf{K}_iu_{r+1}=u_{r+1}, \mathsf{E}_ju_{r+1}=0=\mathsf{F}_ju_{r+1}$ for the Chevalley generators $\mathsf{E}_j, \mathsf{F}_j, 1\leq i\leq r, 1\leq j\leq r-1$ of  $U_q(\mathfrak{gl}_r)$.  
Let $\Phi: U_q(\mathfrak{gl}_r)\rightarrow \End_{\mathbb{C}({v})}(V_{r+1}^{\otimes n})$ the algebra homomorphism which defines the $U_q(\mathfrak{gl}_r)$-module $V_{r+1}^{\otimes n}$.  

Denote by $\breve{\pi}: V_{r+1}\rightarrow V_{r+1}$ be the projection to $V_0$. That is, 
\begin{equation}\label{eq:proj}
\breve{\pi}(u_i)=0, 1\leq i\leq r,\quad \breve{\pi}(u_{r+1})=u_{r+1}. 
\end{equation}
Then by \eqref{eq:restriction} the map $\breve{e}_i:=\underbrace{\breve{\pi}\otimes\breve{\pi}\otimes\cdots\otimes\breve{\pi}}_i\otimes\text{id}\otimes\cdots\otimes\text{id}:V_{r+1}^{\otimes n}\rightarrow V_{r+1}^{\otimes n}$ is actually a $U_q(\mathfrak{gl}_r)$-module homomorphism, thus 
\begin{equation}\label{eq:breve-e}
\breve{e}_i:=\underbrace{\breve{\pi}\otimes\breve{\pi}\otimes\cdots\otimes\breve{\pi}}_i\otimes\text{id}\otimes\cdots\otimes\text{id}\in\End_{U_q(\mathfrak{gl}_r)}(V_{r+1}^{\otimes n}). 
\end{equation}
for $1\leq i\leq n$. Meanwhile define the operator $\breve{R}$ on $V_{r+1}^{\otimes 2}$ via 
\begin{equation}\label{eq:operatorR}
\breve{R}(u_i\otimes u_j)=
\left\{
\begin{array}{ll}
-u_i\otimes u_i,&\text{ if }i=j,\\
-{q}^{\frac{1}{2}}u_j\otimes u_i+(q-1)u_i\otimes u_j,&\text{ if }i>j,\\
-{q}^{\frac{1}{2}}u_j\otimes u_i,&\text{ if }i<j. 
\end{array}
\right. 
\end{equation}
We remark that the operator $\breve{R}$ corresponds to $-q\breve{R}^{-1}_{\rho}$ in \cite[page 466]{Ra} which was originally defined in \cite[Proposition 3, (7)]{Ji}. In addition define 
\begin{equation}\label{eq:operatorRi}
\breve{R}_i=\underbrace{\text{id}\otimes\cdots\otimes\text{id}}_{i-1}\otimes\breve{R}\otimes\underbrace{\text{id}\otimes\cdots\otimes\text{id}}_{n-i-1}\in\End_{\mathbb{C}(q)}(V_{r+1}^{\otimes n})
\end{equation} for $1\leq i\leq n-1$ with $\breve{R}$ acting on the $i$-th and $(i+1)$-th factors. 
It is known \cite{Ji} that 
\begin{equation}\label{eq:R-hom}
\breve{R}_1,\breve{R}_2,\ldots,\breve{R}_{n-1}\in \End_{U_q(\mathfrak{gl}_{r+1})}(V_{r+1}^{\otimes n})
\end{equation}
and moreover they satisfy 
\begin{equation}\label{eq:breveR}
(\breve{R}_i-q)(\breve{R}_i+1)=0,\quad \breve{R}_i\breve{R}_{i+1}\breve{R}_i=\breve{R}_{i+1}\breve{R}_i\breve{R}_{i+1},\quad  \breve{R}_i\breve{R}_j=\breve{R}_j\breve{R}_i
\end{equation}
for $1\leq i\leq n-2, |i-j|>1$
since the map $\psi: \mathcal{H}_n(q)\rightarrow \End_{\mathbb{C}({v})}(V_{r+1}^{\otimes n})$ given by $\psi(T_i)=\breve{R}_i$ for $1\leq i\leq n-1$ defines the $\mathcal{H}_n(q)$-module $V^{\otimes n}_{r+1}$ due to \cite{Ji}. 

\begin{proposition}\label{prop:action-Psi}
The map $\Psi: \mH\rightarrow \End_{\mathbb{C}({v})}(V_{r+1}^{\otimes n})$ defined via 
$$
\Psi(T_i)=\breve{R}_i, \Psi(P_j)=\breve{e}_j,\quad 1\leq i\leq n-1, 1\leq j\leq n
$$
is an algebra homomorphism and moreover $\Psi(\mH)\subseteq \End_{U_q(\mathfrak{gl}_r)}(V_{r+1}^{\otimes n})$. 
\end{proposition}
\begin{proof}
Clearly the second assertion follows from \eqref{eq:breve-e} and \eqref{eq:R-hom}. Then by Theorem \ref{thm:new-pre-R} and \eqref{eq:breveR} it suffices to show the operators $\breve{R}_1,\ldots,\breve{R}_{n-1}, \breve{e}_1,\ldots,\breve{e}_n$ satisfy the relations \eqref{eq:Pi} and \eqref{eq:new-pre-1}-\eqref{eq:new-pre-4}  with $T_i,P_j$ being replaced by $\breve{R}_i,\breve{e}_j$, respectively. It is obvious that $\breve{e}_j^2=\breve{e}_j$ for $1\leq j\leq n$ by \eqref{eq:proj} and \eqref{eq:breve-e}. Hence \eqref{eq:new-pre-1} holds.

Fix $1\leq i<j\leq n$. For each $u=u_{k_1}\otimes u_{k_2}\otimes\cdots\otimes u_{k_n}\in V_{r+1}^{\otimes n}$ with $1\leq k_1,\ldots,k_m\leq r+1$, by \eqref{eq:proj}, \eqref{eq:breve-e}, \eqref{eq:operatorR} and \eqref{eq:operatorRi} a direct calculation gives rise to 
$\breve{R}_i\breve{e}_j(u)=\breve{e}_j\breve{R}_i(u)=-u=-\breve{e}_j(u)$ if $k_1=k_2=\cdots=k_j=r+1$ and $\breve{R}_i\breve{e}_j(u)=\breve{e}_j\breve{R}_i(u)=0=-\breve{e}_j(u)$ otherwise. Hence \eqref{eq:new-pre-4} holds. Similarly one can show \eqref{eq:new-pre-3} holds. 

Finally, it remains to verify \eqref{eq:Pi}  as \eqref{eq:new-pre-2} follows from \eqref{eq:new-pre-1} and \eqref{eq:Pi}. Fix $1\leq i\leq n-1$. For each $u=u_{k_1}\otimes u_{k_2}\otimes\cdots\otimes u_{k_n}\in V_{r+1}^{\otimes n}$ with $1\leq k_1,\ldots,k_m\leq r+1$ a direct calculation by \eqref{eq:operatorR} and \eqref{eq:breve-e} gives rise to 
$-q^{-1}(\breve{e}_i\breve{R}_i\breve{e}_i-(q-1)\breve{e}_i)(u)=0=\breve{e}_{i+1}(u)$ if one of $k_1,\ldots,k_i$ is not equal to $r+1$. Now assume $k_1=k_2=\cdots=k_i=r+1$, then 
\begin{align}
-q^{-1}(\breve{e}_i\breve{R}_i\breve{e}_i-(q-1)\breve{e}_i)(u)
=\left\{
\begin{array}{ll}
u,&\text{ if }k_{i+1}=r+1,\\
0,&\text{ otherwise}
\end{array}
\right.
\end{align}
again by  by \eqref{eq:operatorR} and \eqref{eq:breve-e}. 
Clearly $\breve{e}_{i+1}(u)=u$ if $k_1=k_2=\cdots=k_i=k_{i+1}=r+1$ and $\breve{e}_{i+1}(u)=0$ otherwise. Hence putting together we obtain that \eqref{eq:Pi}  holds. This proves the proposition. 
\end{proof}

\begin{theorem}\label{thm:dulaity}
The algebra homomorphisms $\Psi: \mH\rightarrow \End_{\mathbb{C}({v})}(V_{r+1}^{\otimes n})$ and $\Phi: U_q(\mathfrak{gl}_r)\rightarrow \End_{\mathbb{C}({v})}(V_{r+1}^{\otimes n})$ satisfy the double centralizer property: 
\begin{equation}
\Psi(\mH)= \End_{U_q(\mathfrak{gl}_r)}(V_{r+1}^{\otimes n}), \quad \Phi(U_q(\mathfrak{gl}_r))=\End_{\mH}(V_{r+1}^{\otimes n}). 
\end{equation}
Hence we have the decomposition 
\begin{equation}\label{eq:decompose}
V_{r+1}^{\otimes n}\cong\oplus_{k=0}^n\oplus_{\la\in\mc{P}_k,  \ell(\la)\leq r} L(\la)\otimes N^{(n)}_{(\la,k)}. 
\end{equation}
as $(U_q(\mathfrak{gl}_r),\mH)$-bimodules, where $L(\la)$ is the irreducible $U_q(\mathfrak{gl}_r)$-module of highest weight $\la$. 
\end{theorem}
\begin{proof}
We denote by $\mathscr{C}_n(q)$ the endomorphism algebra 
\begin{equation}\label{eq:centralizer C}
\mathscr{C}_n(q):=\End_{U_q(\mathfrak{gl}_r)}(V_{r+1}^{\otimes n})\cong\End_{\Phi(U_q(\mathfrak{gl}_r))}(V_{r+1}^{\otimes n}). 
\end{equation}
Since $V_{r+1}^{\otimes n}$ is semisimple as a $\Phi(U_q(\mathfrak{gl}_r))$-module, by the general results from double centralizer property (cf. \cite[Section 4.2]{Ha}, see \cite{CR}) we know that $\mathscr{C}_n$ and $\Phi(U_{q}(\mathfrak{gl}_r))$ satisfying the double centralizer property, hence 
$$ 
\End_{\mathscr{C}_n(q)}(V_{r+1}^{\otimes n})=\Phi(U_q(\mathfrak{gl}_r)).
$$ 
We shall apply a strategy similar to  the proof of \cite[Corollary 4.2]{Ha}.  $\Psi(\mH)$ is generated by $\breve{R}_1,\ldots,\breve{R}_{n-1}$ and $\breve{e}_1,\ldots, \breve{e}_n$. Under the specialization $q=1$, these generators specialize exactly to the generators of the endomorphism $\End_{GL(r,\mathbb{C})}(V_{r+1}^{\otimes n})$ and moreover 
$\dim \End_{GL_r(\mathbb{C})}(V_{r+1}^{\otimes n})=\dim\mathscr{C}_n(q)$. Clearly $\{\Psi(T_{(A,B,w)})\mid (A,B,w)\in\Gamma\}$ spans $\Phi(\mH)$ and  $\Psi(T_{(A,B,w)})|_{q=1}$ is well-defined under the specialization $q=1$. This means $\dim \Psi(\mH)\geq\dim \End_{GL(r,\mathbb{C})}(V_{r+1}^{\otimes n}) =\dim\mathscr{C}_n(q)$. Meanwhile  Proposition \ref{prop:action-Psi} and \eqref{eq:centralizer C}  lead to $\Psi(\mH)\subseteq\mathscr{C}_n(q)$. Putting together the theorem is verified. 
\end{proof}
\begin{remark}
Obviously our duality can  fit into the following commutative diagram: 
\begin{equation*}
\begin{array}{ccccc}
&U_q(\mathfrak{gl}_{r+1})\xrightarrow{\hspace{.5cm} \phi \hspace{.5cm}} &\End_{\mathbb{C}({v})}(V_{r+1}^{\otimes n} )&\xleftarrow{\hspace{.5cm} \psi \hspace{.5cm}} &\mathcal{H}_n(q)\\\
&\rotatebox[origin=c]{90}{$\subset$} \qquad \qquad &\rotatebox[origin=c]{90}{$=$}  & &\rotatebox[origin=c]{-90}{$\subset$}\quad \\
 &U_q(\mathfrak{gl}_r)\xrightarrow{\hspace{.5cm} \Phi \hspace{.5cm}} &\End_{\mathbb{C}({v})}(V_{r+1}^{\otimes n} )&\xleftarrow{\hspace{.5cm} \Psi \hspace{.5cm}} &\mH\\
\end{array}
\end{equation*}
Here the first row is the duality in \cite{Ji}.  
Meanwhile it is known that  a Schur-Weyl type duality between $\mH$ and the mirabolic quantum group introduced in \cite{Ro2} (see also \cite{GR})  is established in \cite{FZM} using a geometric method building on the geometric realization of mirabolic Hecke algebra in \cite{Ro1} (see \cite{Ro2} for an algebraic realization of a duality between mirabolic quantum $\mathfrak{sl}_2$ and $\mathscr{H}_2(q)$).  It is interesting to compare the two dualities and one can also ask whether the Schur-Weyl duality in Theorem \ref{thm:dulaity} admits a geometric realization. We define the endomorphism algebra $\mathcal{MS}_q(r,n)=\End_{\mH}(V_{r+1}^{\otimes n})$ to be the {\em mirabolic $q$-Schur algebra} as a natural analog of the classical $q$-Schur algebras $S_q(r,n)$. Clearly $\mathcal{MS}_q(r,n)$ is a subalgebra of classical $q$-Schur algebra $S_q(r+1,n)$. It is natural to ask the connection between our notion of mirabolic $q$-Schur algebra $\mathcal{MS}_q(r,n)$ and the one introduced in \cite{FZM}. 
\end{remark}
\section{Frobenius character formulas for $\mH$}\label{sec:Frobenius}
In this section, we shall apply the duality in Theorem \ref{thm:dulaity} to establish a Frobenius character formula for $\mH$. {\bf In the remainder of the paper, we assume $r>n$}.  

Let $\mathfrak{h}$ be the Cartan subalgebra of the Lie algebra $\mathfrak{gl}_r$ and let $\epsilon_1,\ldots,\epsilon_r$ be the basis of $\mathfrak{h}^*$ due to the standard basis of $\mathfrak{h}$ and set 
$(\epsilon_i,\epsilon_j)=\delta_{ij}$. Let $\mathsf{P}=\sum_{i}\mathbb{Z}\epsilon$ be the weight lattice of $\mathfrak{gl}_r$. Similar to \cite[(3.1)]{Ra},  for formal variables $x_1,x_2,\ldots,x_r$, we introduce the operator $D$ on $U_q(\mathfrak{gl}_r)$-module $V_{r+1}^{\otimes n}$ via 
\begin{equation}\label{eq:operator-D}
D.u_{k_1}\otimes u_{k_2}\otimes\cdots\otimes u_{k_n}=x_{k_1}x_{k_2}\cdots x_{k_n}u_{k_1}\otimes u_{k_2}\otimes\cdots\otimes u_{k_n}, 
\end{equation}
 for any $u=u_{k_1}\otimes u_{k_2}\otimes\cdots\otimes u_{k_n}\in U$ with $k_1,k_2,\ldots,k_n\in\{1,\ldots,r, r+1\}$, where we set $x_{r+1}=1$. Equivalently, for $\mu=\mu_1\epsilon_1+\cdots+\mu_r\epsilon_r\in\mathsf{P}$,  the weight subspace
$(V_{r+1}^{\otimes n})_{\mu}$ of $V_{r+1}^{\otimes n}$ has a basis given by
$\{u_{k_1}\otimes\cdots\otimes u_{k_n}\mid k_1,\ldots,
k_n\in \{1,\ldots,r, r+1\}, \sharp\{j||k_j|=a\}=\mu_a$ for $1\leq a\leq r\}$. The
operator $D$ satisfies 
$$
 D=x_1^{\mu_1}x_2^{\mu_2}\cdots x_r^{\mu_r}~\text{\bf id} \text{ acting on } (V_{r+1}^{\otimes n})_\mu. 
 $$
On the other hand,  by Proposition \ref{prop:action-Psi} the weight space $(V_{r+1}^{\otimes n})_\mu$ is stable under the action of $\mH$ and moreover by \eqref{eq:decompose}  it can be decompose as the direct sum of irreducible modules $N^{(n)}_{\la,k}$ as follows
$$
(V_{r+1}^{\otimes
n})_{\mu}\cong \oplus_{k=0}^n\oplus_{\la\in\mathcal{P}_k} L(\la)_\mu\otimes N^{(n)}_{(\la,k)}
$$
under the assumption $r>n$. 
This implies that the trace of $Dh$ on $(V_{r+1}^{\otimes n})_{\mu}$ can
be written as
$$
{\rm tr} Dh|_{(V_{r+1}^{\otimes n})_{\mu}}=\sum_{k=0}^n\sum_{\la\in\mc{P}_k}\dim(L(\la)_\mu)x_1^{\mu_1}\cdots
x_r^{\mu_r}\chi^{(n)}_{(\la,k)}(h)
$$
This means as operators on $V_{r+1}^{\otimes n}$: 
 \begin{equation}\label{eq:trace-duality}
 \begin{aligned}
{\rm tr} (Dh)=&\sum_{\mu\in\mathsf{P}}\sum_{k=0}^n\sum_{\la\in\mc{P}_k}\dim(L(\la)_\mu)x_1^{\mu_1}\cdots
x_r^{\mu_r}\chi^{(n)}_{(\la,k)}(h)\\
=&\sum_{k=0}^n\sum_{\la\in\mc{P}_k}{\rm ch}(L(\la))\chi^{(n)}_{(\la,k)}(h)=\sum_{k=0}^n\sum_{\la\in\mc{P}_k}s_\la(x_1,\ldots,x_r)\chi^{(n)}_{(\la,k)}(h). 
\end{aligned}
\end{equation}
due to the known result ${\rm ch}(L(\la)):=\sum_{\mu\in\mathsf{P}}\dim(L(\la)_\mu) x_1^{\mu_1}\cdots
x_r^{\mu_r}=s_\la(x_1,\ldots,x_r)$ (cf \cite[Proposition 10.1.5]{CP}), where $s_\la(x_1,\ldots,x_r)$ is the Schur function in $x_1,\ldots,x_r$ corresponding to the partition $\la$. 

 In the following, we shall give another way to compute the trace ${\rm tr}(Dh)$ for $h=\widehat{T}^{(n)}_\la$ with $\la\in\mathcal{P}_k, 0\leq k\leq n$ by applying the approach in \cite{Ra} (see also \cite{WW} for the case Hecke-Clifford algebras). 

\begin{lemma}\label{lem:trace-wn}
For $1\leq m\leq n$, the trace of the operator $DT_{w_{(m)}}$ on $V_{r+1}^{\otimes m}$ is given by 
$$
{\rm tr}(DT_{w_{(m)}})=\sum_{\underline{k}}(-1)^{\mathsf{e}(\underline{k})}(q-1)^{\mathsf{g}(\underline{k})} x_{k_1}x_{k_2}\cdots x_{k_m}, 
$$
where the summation is over $\underline{k}=(k_1,k_2,\ldots,k_m)$ such $r+1\geq k_1\geq k_2\geq k_m\geq 1$ and $\mathsf{e}(\underline{k})=\sharp\{a\mid k_a=k_{a+1}\},\mathsf{g}(\underline{k})=\sharp\{a\mid k_a>k_{a+1}\}$ and we recall the convention $x_{r+1}=1$. 
\end{lemma}
\begin{proof}
It suffices to show that, for $u=u_{k_1}\otimes u_{k_2}\otimes\cdots\otimes u_{k_m}$,
\begin{align*}
(DT_{w_{(m)}} u)|_u
=\left\{
\begin{array}{ll}
(-1)^{\mathsf{e}(\underline{k})}(q-1)^{\mathsf{g}(\underline{k})}x_{k_1}\cdots x_{k_m},
&\text{ if } r+1\geq k_1\geq k_2\geq\cdots\geq k_m\geq 1,\\
0, &\text{ otherwise}.
\end{array}
\right.
\end{align*}
By definition,  we have $T_{w_{(m)}} =T_{w_{(m-1)}} T_{m-1}$.
It follows by Proposition \ref{prop:action-Psi} and \eqref{eq:operator-D} that
\begin{align*}
&(DT_{w_{(m)}} u)|_u\\
=&x_{k_1}\cdots x_{k_m}\cdot (T_{w_{(m)}} u)|_u\\
=&x_{k_1}\cdots x_{k_m}\cdot T_{w_{(m-1)}}
\big(u_{k_1}\otimes\cdots\otimes u_{k_{m-2}}\otimes
 \breve{R}  (u_{k_{m-1}}\otimes u_{k_m}) \big)|_u\\
=&\left\{
\begin{array}{llll}
- x_{k_1}\cdots x_{k_m}  T_{w_{(m-1)}}(u_{k_1}\otimes\cdots\otimes
u_{k_{m-1}})|_{u_{k_1}\otimes\cdots\otimes u_{k_{m-1}}},
&\text{ if }k_{m-1}=k_m\geq 1,\\
(q-1) x_{k_1}\cdots x_{k_m}
 T_{w_{(n-1)}}(u_{k_1}\otimes\cdots\otimes
 u_{k_{n-1}})|_{u_{k_1}\otimes\cdots\otimes u_{k_{m-1}}},
&\text{ if }k_{m-1}>k_m,\\
0,&\text{ otherwise}.
\end{array}
\right.
\end{align*}
Now the lemma follows by induction on $m$.
\end{proof} 
By Lemma \ref{lem:trace-wn} the following is straightforward. 
\begin{proposition}\label{prop:trace-wn}
For $1\leq m\leq n$, the trace of the operator $DT_{w_{(m)}}$ on $V_{r+1}^{\otimes m}$ satisfies 
$$
{\rm tr}(DT_{w_{(m)}})=\sum_{\mu\in\mathcal{P}_m}(-1)^{m-\ell(\mu)}(q-1)^{\ell(\mu)-1}m_\mu(x_1,\ldots,x_r,x_{r+1}). 
$$
\end{proposition}
\begin{proof}
Given  a composition $\alpha=(\alpha_1,\ldots,\alpha_{r},\alpha_{r+1})$ of $m$, the $\underline{k}=(k_1,\ldots, k_m)$ satisfying $r+1\geq k_1\geq k_2\geq\cdots\geq k_r\geq 1$ and
$\sharp\{a\mid k_a=j\}=\alpha_j$ for $1\leq j\leq r+1$ is exactly $\underline{k}=(r+1,\ldots, r+1, r,\ldots,r,\ldots,1,\ldots,1)$. By
Lemma~\ref{lem:trace-wn}, the coefficient  of the monomial
$x^{\alpha}:=x_1^{\alpha_1}\cdots x_r^{\alpha_r}x_{r+1}^{\alpha_{r+1}}$ in ${\rm
tr}(DT_{w_{(m)}})$, denoted by ${\rm
tr}(DT_{w_{(m)}})|_{x^{\alpha}}$,  is given by
$$
{\rm tr}(DT_{w_{(m)}})|_{x^{\alpha}}
=(-1)^{m-\ell(\alpha)}(q-1)^{\ell(\alpha)-1}. 
$$
Then the proposition follows. 
\end{proof}
For $m\geq 0$, a parameter $t$ and $x =(x_1,\ldots, x_r,x_{r+1})$ with $x_{r+1}=1$, the Hall-Littlewood symmetric functions $q_m(x;t)$ is defined  by
the following generating function in a variable $y$:
\begin{align}\label{eq:generating}
\sum_{m\geq 0}q_m(x;t)y^m=\prod_{i}
\frac{1 -tx_iy}{1 -x_iy},
\end{align}
and set
\begin{equation}\label{eq:onecharacter}
\widetilde{q}_m(x;t)=\frac{(-1)^m}{t-1}q_m(x;t).
\end{equation}
\begin{proposition}\label{prop:trace-q-function}
For $m\geq 1$, we have ${\rm tr}(DT_{w_{(m)}})=\widetilde{q}_m(x;q)$. 
\end{proposition}
\begin{proof}
Observe that 
$$
\prod_{i} 
\sum_{m\geq 0}\sum_{\mu\in\mathcal{P}_m}(1-t)^{\ell(\mu)}m_{\mu}(x)y^m=\prod_{i}(1+(1-t)x_iy+(1-t)x_i^2y^2+\cdots)=\frac{1 -tx_iy}{1 -x_iy}
$$
and hence by \eqref{eq:generating} we have $q_m(x;t)=\sum_{\mu\in\mathcal{P}_m}(1-t)^{\ell(\mu)}m_{\mu}(x)$. Meanwhile by Proposition \ref{prop:trace-wn} we have 
$$
{\rm tr}(DT_{w_{(m)}})=\frac{(-1)^m}{q-1}\sum_{\mu\in\mathcal{P}_m}(1-q)^{\ell(\mu)}m_{\mu}(x). 
$$
Hence the proposition is proved by \eqref{eq:onecharacter}. 
\end{proof}
For a partition $\mu=(\mu_1,\ldots,\mu_{\ell})$, we define
\begin{align*}
\widetilde{q}_{\mu}(x;q)=\widetilde{q}_{\mu_1}(x;q)\widetilde{q}_{\mu_2}(x;q)\cdots
\widetilde{q}_{\mu_{\ell}}(x;q).
\end{align*}
We are ready to establish a Frobenius type formula for the
characters $\chi^{(n)}_{\la,k}$ of the mirabolic Hecke algebra $\mH$.

\begin{theorem}\label{thm:Frobenius}
The following holds for each partition $\mu$ of $k$ with $0\leq k\leq n$:
\begin{equation}\label{eq:Frobenius}
\widetilde{q}_{\mu}(x_1,x_2,\ldots,x_r,1;q)
=\sum_{\la\in\mc{P}_k, 0\leq k\leq n}
s_{\la}(x_1,\ldots,x_r)\chi^{(n)}_{(\la,k)}(\widehat{T}^{(n)}_{\mu}).
\end{equation}
\end{theorem}
\begin{proof}
Suppose $\mu=(\mu_1,\mu_2,\ldots,\mu_\ell)\in\mathcal{P}_k$ for some $0\leq k\leq n$. Recall that $\widehat{T}^{(n)}_\mu=P_{n-k}T_{{\mu^{\uparrow n}},n-k+1}T_{{\mu^{\uparrow n}},n-k+2}\cdots T_{{\mu^{\uparrow n}},n-k+\ell}$ by \eqref{eq:hat-T-mu}. As the elements $P_{n-k}$, $T_{{\mu^{\uparrow n}},n-k+1},$ $ \ldots, T_{{\mu^{\uparrow n}},n-k+\ell}$ commute with each other and moreover $P_{n-k}$ acts on the first $n-k$ factors, $T_{{\mu^{\uparrow n}},n-k+1}$ acts on the subsequent $\mu_1$ factors and so on. So 
\begin{equation}\label{eq:trace-prod}
{\rm tr}(D\widehat{T}^{(n)}_\mu)= {\rm tr}(DP_{n-k})\cdot {\rm tr}(DT_{{\mu^{\uparrow n}},n-k+1})\cdots {\rm tr}(DT_{{\mu^{\uparrow n}},n-k+\ell})\
\end{equation}
By Proposition \ref{prop:action-Psi} and \eqref{eq:breve-e} we observe that the action of $DP_{n-k}$ on the first $n-k$ factors has trace equal to $1$. This together with Proposition \ref{prop:trace-q-function} and \eqref{eq:trace-prod} leads to ${\rm tr}(D\widehat{T}^{(n)}_\mu)=\widetilde{q}_{\mu_1}(x;q)\widetilde{q}_{\mu_2}(x;q)\cdots
\widetilde{q}_{\mu_{\ell}}(x;q)=\widetilde{q}_{\mu}(x;q)$. Then by \eqref{eq:trace-duality} the theorem is verified. 
\end{proof}

\begin{corollary}
Suppose $\la,\mu\in\cup_{t=0}^n\mathcal{P}_t$ and set $k=|\la|$. Then $\chi^{(n)}_{(\la,k)}(\widehat{T}^{(n)}_\mu)=0$ unless $|\la|\leq |\mu|$. 
\end{corollary}
\section{Murnaghan-Nakayama rule for characters of $\mH$}\label{sec:MN}
In this section, we shall provide a combinatorial algorithm to compute the characters $\chi^{(n)}_{\la,k}(\widehat{T}_{\nu}).$ For two partitions $\la,\nu$, we say $\nu\subseteq\la$ of $\nu_i\leq \la_i$ for all $i$. 

In the standard fashion (see \cite{Mac}), to each partition $\la$ we 
associate a diagram. If $\nu\subseteq\la$, we have the skew shape, denote by $\la/\nu$, which is the set 
theoretic difference of the diagrams $\la$ and $\nu$. A horizontal (resp. vertical) strip is 
a skew diagram with at most one box in each column (row). A strip is a skew 
diagram that does not contain any $2 \times 2$ block of boxes. Two boxes are connected if 
they have an edge in common. Any strip is 
a union of connected components.  For a skew shape $\la/\nu$, following \cite{Ra} we set 
$$
{\rm wt}_{\la/\nu}(q)=\left\{
\begin{array}{cc}
(q-1)^{\mathsf{cc}(\la/\nu)-1}\prod_{\mathsf{b}}q^{\mathsf{co}(b)-1}(-1)^{\mathsf{ro}(b)-1}, &\text{ if }\la/\nu\text{ is a strip}, \\
0, &\text{ otherwise},
\end{array}
\right. 
$$
where the product is over all connected components $\mathsf{b}$ of $\la/\nu$, $\mathsf{cc}(\la/\nu)$ is the number of connected components in $\la/\nu$, $\mathsf{co}(b)$ (resp. $\mathsf{ro}(b)$) is the number of columns (resp. rows) that $\mathsf{b}$ occupies. Accordingly, we have 
\begin{equation}\label{eq:wt-strip}
\begin{aligned}
&\overline{\rm wt}_{\la/\nu}(q):={\rm wt}_{\la/\nu}(q^{-1})\\
=&\left\{
\begin{array}{cc}
(-q)^{-|\la/\nu|+1}(q-1)^{\mathsf{cc}(\la/\nu)-1}\prod_{\mathsf{b}}q^{\mathsf{ro}(b)-1}(-1)^{\mathsf{co}(b)-1}, &\text{ if }\la/\nu\text{ is a strip}, \\
0, &\text{ otherwise},
\end{array}
\right. 
\end{aligned}
\end{equation}
since we observe $(\mathsf{cc}(\la/\nu)-1)+\sum_{\mathsf{b}}(\mathsf{ro}(b)-1)+\sum_{\mathsf{b}}(\mathsf{co}(b)-1)=|\la/\nu|-1$. 

For the formal variable $y_1,y_2,\ldots,y_{r+1}$, set
\begin{align*}
\widetilde{q}_m(y_1,\ldots,y_r,y_{r+1};q)=\sum_{\underline{k}}(-1)^{\mathsf{e}(\underline{k})}(q-1)^{\mathsf{g}(\underline{k})} y_{k_1}y_{k_2}\cdots y_{k_m}
\end{align*}
where the first summation is over $\underline{k}=(k_1,k_2,\ldots,k_m)$ such $r+1\geq k_1\geq k_2\geq k_m\geq 1$ and $\mathsf{e}(\underline{k})=\sharp\{a\mid k_a=k_{a+1}\},\mathsf{g}(\underline{k})=\sharp\{a\mid k_a>k_{a+1}\}$. Meanwhile, let $g_0(y_1,y_2,\ldots,y_{r+1};q)=1$ and for $m\geq 1$ set (see \cite[(2.9)]{DHP})
\begin{equation*}
g_m(y_1,y_2,\ldots,y_{r+1};q)=\sum_{\underline{i}}q^{{\mathsf{e}'}(\underline{i})}(q-1)^{{\mathsf{g}'}(\underline{i})}y_{i_1}y_{i_2}\cdots y_{i_m}, 
\end{equation*}
where the summation is over $\underline{i}=(i_1,i_1,\ldots,i_m)$ such $1\leq i_1\leq i_2\leq\cdots\leq i_m\leq r+1$, ${\mathsf{e}'}(\underline{i})=\sharp\{b\mid i_b=i_{b+1}\}$ and  ${\mathsf{g}'}(\underline{i})=\sharp\{b\mid i_b<i_{b+1}\}$.  
\begin{lemma} \label{lem:two-symmetric}
For $m\geq 0$, we have $\widetilde{q}_m(y_1,\ldots,y_r,y_{r+1};q)=(-q)^{m-1}g_m(y_1,\ldots,y_r,y_{r+1};q^{-1}).$
\end{lemma}
\begin{proof}
By  \cite[Theorem 4.1, Theorem 4.3]{Ra} or by an argument similar to the proof of Proposition \ref{prop:trace-wn} the following holds: 
\begin{equation*}
\begin{aligned}
g_m(y_1,y_2,\ldots,y_{r+1};q^{-1})=&\sum_{\mu\in\mathcal{P}_m}(q^{-1})^{m-\ell(\mu)}(q^{-1}-1)^{\ell(\mu)-1}m_\mu(y_1,y_2,\ldots,y_{r+1})\\
=&(-q)^{-m+1} \sum_{\mu\in\mathcal{P}_m}(-1)^{m-\ell(\mu)}(q-1)^{\ell(\mu)-1}m_\mu(y_1,y_2,\ldots,y_{r+1})\\
=&(-q)^{-m+1}\widetilde{q}_m(y_1,\ldots,y_r,y_{r+1};q), 
\end{aligned}
\end{equation*}
where the last equality is due to Proposition \ref{prop:trace-wn} and Proposition \ref{prop:trace-q-function}. Then the lemma is proved. 
\end{proof}
The following is due to  \cite{DHP} :
\begin{lemma}\cite[Proposition 3.3]{DHP}\label{lem:DHP} 
Suppose $0\leq m\leq n$ and $\nu\in\cup_{k=0}^{n-m}\mathcal{P}_{k}$. Then 
$$
g_m(1,x_1,\ldots,x_r;q)s_\nu(x_1,\ldots,x_r)=\sum_{\la}\mathsf{f}_{|\la/\nu|,m}(q){\rm wt}_{\la/\nu}(q)s_\la(x_1,\ldots,x_r). 
$$
where the summation is over the partitions $\la\in\cup_{k=0}^n\mathcal{P}_k$ such that $\la/\nu$ is a strip with $0\leq |\la/\nu|\leq m$ and  $\mathsf{f}_{t,m}(q)$ for $0\leq t\leq m$ is defined via 
 \begin{equation}\label{eq:f-number}
 \mathsf{f}_{t,m}(q)=
 \left\{
 \begin{array}{ll}
 q^{m-1},&\text{ if }t=0,\\
(q-1)q^{m-t-1},&\text{ if }0<t<m,\\
 1,&\text{ if }t=m. 
 \end{array}
 \right.
 \end{equation}
\end{lemma}

\begin{proposition} \label{prop:new-DHP}
Suppose $0\leq m\leq n$ and $\nu\in\cup_{k=0}^{n-m}\mathcal{P}_{k}$, we have 
$$
\widetilde{q}_m(x_1,\ldots,x_r,1;q)s_\nu(x_1,\ldots,x_r)=\sum_{\la} \mathsf{g}_{|\la/\nu|,m}(q)\overline{{\rm wt}}_{\la/\nu}(q)s_\la(x_1,\ldots,x_r),
$$
 where the summation is over the partitions $\la\in\cup_{k=0}^n\mathcal{P}_k$ such that $\la/\nu$ is a strip with $0\leq |\la/\nu|\leq m$ and  $\mathsf{g}_{t,m}(q)$ for $0\leq t\leq m$ is defined via 
 \begin{equation}\label{eq:g-number}
 \mathsf{g}_{t,m}(q)=(-q)^{m-1}\mathsf{f}_{t,m}(q)=
 \left\{
 \begin{array}{ll}
 (-1)^mq,&\text{ if }t=0,\\
 (-1)^{m-t+1}(q-1),&\text{ if }0<t<m,\\
 1,&\text{ if }t=m. 
 \end{array}
 \right.
 \end{equation}
 \end{proposition}
 \begin{proof}
 By Lemma \ref{lem:two-symmetric} we have 
 \begin{align*}
 \widetilde{q}_m(x_1,\ldots,x_r,1;q)=&(-q)^{m-1}g_m(x_1,x_2,\ldots,x_r,1;q^{-1})\\
 =&(-q)^{m-1}g_m(1, x_1,x_2,\ldots,x_r;q^{-1}). 
 \end{align*}
 Then the proposition follows from Lemma \ref{lem:DHP}. 
 \end{proof}
 
 It is ready to state the Murnaghan-Nakayama rule for computing the irreducible characters for $\mH$. 
\begin{theorem}\label{thm:MN rule}
Let $0\leq k\leq n$ and $\la\in\mathcal{P}_k$. Assume $\mu=(\mu_1,\mu_2,\ldots,\mu_\ell)\in\cup_{t=0}^n\mathcal{P}_t$ with $\ell(\mu)=\ell$. Let $\mu_\ell=m$ and $\mu^-=(\mu_1,\mu_2,\ldots,\mu_{\ell-1})$. Then 
$$
\chi^{(n)}_{(\la,k)}(\widehat{T}^{(n)}_\mu)= \sum_{\nu}\mathsf{g}_{|\la/\nu|,m}(q)~\overline{{\rm wt}}_{\la/\nu}(q)~\chi^{(n-m)}_{(\nu,t)}(\widehat{T}^{(n-m)}_{\mu^{-}}), 
$$
where $\overline{{\rm wt}}_{\la/\nu}(q)$  and $\mathsf{g}_{|\la/\nu|,m}(q)$ are given in \eqref{eq:wt-strip} and \eqref{eq:g-number}, and the summation is over $\nu\in\mathcal{P}_t, 0\leq t\leq n-m$ such that $\la/\nu$ is a strip and $0\leq |\la/\nu|\leq m$. 
\end{theorem}

\begin{proof}
By Theorem \ref{thm:Frobenius} and Lemma \ref{lem:DHP} we obtain 
\begin{align*}
&\sum_{\la\in\mc{P}_k, 0\leq k\leq n}
s_{\la}(x_1,\ldots,x_r)\chi^{(n)}_{(\la,k)}(\widehat{T}^{(n)}_{\mu})\\
&=\widetilde{q}_{\mu}(x_1,x_2,\ldots,x_r,1;q)\\
&=\widetilde{q}_{\mu^-}(x_1,x_2,\ldots,x_r,1;q)\widetilde{q}_t(x_1,x_2,\ldots,x_r,1;q)\\
&=\sum_{\nu\in\mathcal{P}_t,0\leq t\leq n-m}\chi^{(n-m)}_{(\nu,t)}(\widehat{T}^{(n-m)}_{\mu^-})s_\nu(x_1,\ldots,x_r)\widetilde{q}_t(x_1,x_2,\ldots,x_r,1;q)\\
&=\sum_{\nu\in\mathcal{P}_t,0\leq t\leq n-m}\sum_{\substack{\la\in\cup_{k=0}^n\mathcal{P}_k\\ \la/\nu\text{ is a strip}, 0\leq |\la/\nu|\leq n-m}}~\chi^{(n-m)}_{(\nu,t)}(\widehat{T}^{(n-m)}_{\mu^-}) ~\mathsf{g}_{|\la/\nu|,m}(q)~\overline{{\rm wt}}_{\la/\nu}(q)s_\la(x_1,\ldots,x_r). 
\end{align*}
As the Schur polynomials are linearly independent, the theorem is proved. 
\end{proof}

\end{document}